\documentclass[final]{myarticle}

\calclayout
\numberwithin{equation}{section}
\allowdisplaybreaks
\theoremstyle{plain}

\newtheorem{Thm}{Theorem}[section]
\newtheorem{Lemma}[Thm]{Lemma}

\newtheorem{Prop}[Thm]{Proposition}
\newtheorem{Rem}[Thm]{Remark}

\newtheorem{Def}[Thm]{Definition}

\usepackage{enumerate}
\usepackage{amsmath}
\usepackage{amssymb}
\usepackage{mathtools}
\mathtoolsset{showonlyrefs}
\usepackage{mathrsfs}
\usepackage{comment}
\usepackage{hyperref}
\usepackage{xcolor}
\usepackage{extarrows}
\usepackage{graphicx}
\usepackage{hyperref}
\usepackage{subfigure}
\usepackage{doi}
\usepackage{cases}
\usepackage{tikz}
\usetikzlibrary{positioning, shapes.geometric}
\usepackage{dsfont}

\def\I{\mathcal{I}}
\def\i{\mathrm{i}}

\begin{document}
	
\title{Optimal Rate of Convergence for Vector-valued Wiener-It\^o Integral}

\author{Huiping \textsc{Chen}}
\address{LMAM, School of Mathematical Sciences, Peking University, Beijing 100871, China}
\email{chenhp@pku.edu.cn}


\keywords{Optimal rate of convergence, vector-valued Wiener-It\^o integral, Malliavin calculus, Stein's method, method of cumulants. \\\indent 	
	This work is supported by NSFC (No. 11731009, No. 12231002) and Center for Statistical Science, PKU. The author is extremely grateful to Prof. Yong Chen and Prof. Yong Liu for valuable comments and discussions}
	
\begin{abstract}
	We investigate the optimal rate of convergence in the multidimensional normal approximation of vector-valued Wiener-It\^o integrals of which components all belong to the same fixed Wiener chaos. Combining Malliavin calculus, Stein's method for normal approximation and method of cumulants, we obtain the optimal rate of convergence with respect to a suitable smooth distance. As applications, we derive the optimal rates of convergences for complex Wiener-It\^o integrals, vector-valued Wiener-It\^o integrals with kernels of step functions and vector-valued Toeplitz quadratic functionals.
\end{abstract}
	\maketitle

	\section{Introduction}
	On a complete probability space $(\Omega, \mathcal{F}, P)$, let $X=\left\{ X(h): h\in\mathfrak{H}\right\} $ be an isonormal Gaussian process over some real separable Hilbert space $\mathfrak{H}$, where the $\sigma$-algebra $\mathcal{F}$ is generated by $X$. Let $\left\lbrace F_n: n\geq 1\right\rbrace $ be a sequence of random variables living in a fixed Wiener chaos of $X$ with unit variance. In recent years, the research associated with the normal approximation of $\left\lbrace F_n: n\geq 1\right\rbrace $ has always been concerned. In 2005, Nualart and Peccati published the seminal article \cite{NP2005} and first proved Fourth Moment Theorem which shows that $\left\lbrace F_n: n\geq 1\right\rbrace $ converges to a standard normal random variable $N$ if and only if  $\mathrm{E}\left[ F_n^4\right]\rightarrow 3 $ as $n\rightarrow\infty$. Shortly afterwards, a multidimensional version of this characterization was given by Peccati and Tudor in \cite{Peccati2005}. By using techniques of Malliavin calculus, Nualart and Ortiz-Latorre proposed a new proof of Fourth Moment Theorem in \cite{NUALART2008}. Further, in \cite{NP2009}, Nourdin and Peccati combined Malliavin Calculus with Stein's method to derive quantitative and explicit upper bounds in the Gaussian approximation of $\left\lbrace F_n: n\geq 1\right\rbrace $. In this paper, we focus on the optimal rate of convergence with respect to a suitable distance under the assumption that $F_n$ converges to $N$ in distribution. We say that a positive sequence $\left\lbrace \varphi(n):n\geq1 \right\rbrace$ decreasing to zero provides an optimal rate of convergence with respect to some distance $d(\cdot,\cdot)$, if $d(F_n,N)\asymp \varphi(n)$. Here, for two numerical sequences $\left\lbrace a_n:n\geq1 \right\rbrace $ and $\left\lbrace b_n:n\geq1 \right\rbrace $, we write $a_n\asymp b_n$ if there exist two constants $0<c_1<c_2<\infty$ not depending on $n$ such that $c_1b_n\leq a_n\leq c_2b_n$ for $n$ sufficiently large. Throughout the paper, we denote by $c_1$ and $c_2$ two finite positive constants that not depend on $n$ and can vary from line to line.

	Fix an integer $q\geq2$. For a sequence of random variables $\left\lbrace F_n=I_q(f_n): n\geq 1\right\rbrace $ with unit variance and all $f_n\in \mathfrak{H}^{\odot p}$, assume that $F_n$ converges to $N$ in distribution. There are complete characterisation of optimal rate of convergence with respect to some suitable distance $d(\cdot,\cdot)$.
	In \cite{NP2009b}, Nourdin and Peccati showed how to detect optimal Berry-Esseen bounds in the normal approximation of functionals of $X$ and further refined the main results they proven in \cite{NP2009}. Specifically, they supposed additionally that as $n \rightarrow \infty$, the two-dimensional random vector 
	\begin{equation}\label{Additional condition}
		\left( F_n, \frac{1-q^{-1}\left\| DF_n\right\|^{2}_{\mathfrak{H}} }{\sqrt{\mathrm{Var}\left( q^{-1}\left\| DF_n\right\|^{2}_{\mathfrak{H}}\right) }}\right) \overset{d}{\rightarrow}(N_1,N_2),
	\end{equation}
 as $n \rightarrow \infty$, where $\left(N_1, N_2\right)$ is a centered two-dimensional Gaussian vector satisfying $\mathrm{E}\left(N_1^2\right)=\mathrm{E}\left(N_2^2\right)=1$ and $\mathrm{E}\left(N_1 N_2\right)=\rho$. If $\rho\neq0$, then $\frac{P(F_n\leq z)-P(N\leq z)}{\sqrt{\mathrm{Var}\left( q^{-1}\left\|  D F_n\right\| _{\mathfrak{H}}^2\right) }}$ converges to a nonzero limit for every $z\in\mathbb{R}$. Therefore,
 \begin{equation}\label{Rate of NP09}
 	d_{\mathrm{Kol}}\left(F_n, N\right)\asymp \sqrt{\mathrm{Var}\left( q^{-1}\left\|  D F_n\right\| _{\mathfrak{H}}^2\right) },
 \end{equation}
 where $d_{\mathrm{Kol}}\left(F_n, N\right)$ is Kolmogorov distance defined as $$d_{\mathrm{Kol}}\left(F_n, N\right)=\sup\limits_{z\in\mathbb{R}}\left|P(F_n\leq z)- P(N\leq z)\right|.$$
 Note that $\sqrt{\mathrm{Var}\left( q^{-1}\left\|  D F_n\right\| _{\mathfrak{H}}^2\right) }\asymp\sqrt{\mathrm{E}\left[F_n^4 \right]-3 }$ (see \cite[Lemma 5.2.4]{nourdin2012normal}). In \cite[Proposition 3.6]{NP2009b}, they proposed that, if $q$ is even, sufficient conditions for \eqref{Additional condition} are as $n \rightarrow \infty$,
	\begin{equation}\label{Condition 1}
		\sum_{r=1}^{q-1}\sum_{l=1}^{2(q-r)-1}\frac{\left\|\left(f_n \tilde{\otimes}_r f_n\right) \otimes_l\left(f_n \tilde{\otimes}_r f_n\right)\right\|_{\mathfrak{H} ^{\otimes 2(2(q-r)-l)}}}{\mathrm{Var}\left( q^{-1}\left\|  D F_n\right\| _{\mathfrak{H}}^2\right) } \rightarrow 0,
	\end{equation} 
 and
	\begin{equation}\label{Condition 2}
		-q q !(q / 2-1) !\binom{q-1}{q/2-1}^2 \frac{\left\langle f_n, f_n \tilde{\otimes}_{q / 2} f_n\right\rangle_{\mathfrak{H}^{\otimes q}}}{\sqrt{\mathrm{Var}\left( q^{-1}\left\|  D F_n\right\| _{\mathfrak{H}}^2\right) }} \rightarrow \rho.
	\end{equation}
In this case, if $\rho \neq 0$, then \eqref{Rate of NP09} is valid, that is, $\sqrt{\mathrm{Var}\left( q^{-1}\left\|  D F_n\right\| _{\mathfrak{H}}^2\right) }$ is the optimal rate of convergence for $F_n$ with respect to Kolmogorov distance. However, when $q$ is even and $\rho=0$, or, $q$ is odd and \eqref{Condition 1} is satisfied (which imply \eqref{Additional condition} with $\rho=0$), the optimal rate of convergence with respect to Kolmogorov distance is unknown. In \cite{Hermine2012}, Bierm\'e, Bonami, Nourdin and Peccati gave a complete solution to the optimal rate of convergence in the case of a suitable smooth distance. They proved that $$d\left(F_n, N\right)\asymp\max\left(\left|\mathrm{E}\left[F_n^3\right]\right|,\mathrm{E}\left[F_n^4\right]-3\right),$$ where $d\left(F_n, N\right)=\sup \left|\mathrm{E}\left[h\left(F_n\right)\right]-\mathrm{E}[h(N)]\right|$, and $h$ runs over the class of all real functions with a second derivative bounded by one. Note that, it is shown in \cite[Prposition 3.1]{NP2010} that $$\left|\mathrm{E}\left[F_n^3\right]\right|\leq c\sqrt{\mathrm{E}\left[F_n^4\right]-3},$$ where $c$ is constant only depending on $q$. Furthermore, in \cite{NP2015}, Nourdin and Peccati obtained that $\left\lbrace \max \left(\left|\mathrm{E}\left[F_n^3\right]\right|, \mathrm{E}\left[F_n^4\right]-3\right):n\geq1\right\rbrace$ also provides an optimal rate of convergence in total variation, a non-smooth distance. That is, 
	\begin{equation}
		d_{\mathrm{TV}}\left(F_n, N\right)\asymp \max \left(\left|\mathrm{E}\left[F_n^3\right]\right|, \mathrm{E}\left[F_n^4\right]-3\right),
	\end{equation}
	where $d_{\mathrm{TV}}\left(F_n, N\right)$ is total variation distance defined as $$d_{\mathrm{TV}}\left(F_n, N\right)=\sup\limits_{A\in\mathcal{B}(\mathbb{R})}\left|P(F_n\in A)- P(N\in A)\right|.$$

	As far as we know, there are few references studying the optimal rate of convergence for a sequences of random vectors of which components are functionals of some isonormal Gaussian process. In \cite{Campese2013}, Campese extended the results of \cite{NP2009b} to the multidimensional case and developed techniques for determining the exact asymptotic speed of
	convergence in the multidimensional normal approximation of smooth functionals of isonormal Gaussian processes. Let $\left\{F_n=(F_{n,1}, \ldots,F_{n,d}): n \geq 1\right\}$ be a sequence of $d$-dimensional random vectors with $F_{n,i}=I_{q_i}(f_{n,i})$ and $f_{n,i}\in\mathfrak{H}^{\odot q_i}$ for $1\leq i\leq d$. Suppose that the covariance matrix of $F_n$ is $C$ and $F_n$ converges in distribution to $d$-dimensional normal random vector $Z \sim \mathscr{N}_d(0,C)$. Analogously to the one-dimensional case, the random sequence
	\begin{equation}\label{Sequence}
		\left\lbrace \left( F_n, \frac{q_j^{-1}\left\langle DF_{n,i},DF_{n,j}\right\rangle _{\mathfrak{H}}-\mathrm{E}\left[q_j^{-1}\left\langle DF_{n,i},DF_{n,j}\right\rangle _{\mathfrak{H}} \right]}{\sqrt{\mathrm{Var}\left(q_j^{-1}\left\langle DF_{n,i},DF_{n,j}\right\rangle _{\mathfrak{H}}\right) }} \right): n\geq1 \right\rbrace 
	\end{equation}
	plays a crucial role. To be detailed, let $g: \mathbb{R}^d \rightarrow \mathbb{R}$ be three-times differentiable with bounded derivatives up to order three. Suppose that for $1 \leq i, j \leq d$, the random sequences \eqref{Sequence} converge in law to a centered Gaussian random vector $\left(Z, \tilde{Z}_{i j}\right)$ whenever
	\begin{equation}\label{Condition 3}
    	\sqrt{\mathrm{Var}\left(q_j^{-1}\left\langle DF_{n,i},DF_{n,j}\right\rangle _{\mathfrak{H}}\right) } \asymp \sqrt{\sum_{i,j=1}^{d}\mathrm{Var}\left(q_j^{-1}\left\langle DF_{n,i},DF_{n,j}\right\rangle _{\mathfrak{H}}\right) }.
	\end{equation}
	\cite[Theorem 3.4, Corollary 3.6]{Campese2013} show that the $\liminf$ and $\limsup$ of the sequence 
	\begin{equation}
		\left\lbrace \frac{\mathrm{E}\left[g\left(F_n\right)\right]-\mathrm{E}[g(Z)]}{\sqrt{\sum_{i,j=1}^{d}\mathrm{Var}\left(q_j^{-1}\left\langle DF_{n,i},DF_{n,j}\right\rangle _{\mathfrak{H}}\right) }}: n\geq 1\right\rbrace  
	\end{equation}
 coincide with those of 
 \begin{equation}\label{Condition 4}
 	\left\lbrace \frac{1}{3} \sum_{i, j, k=1}^d \frac{\sqrt{\mathrm{Var}\left(q_j^{-1}\left\langle DF_{n,i},DF_{n,j}\right\rangle _{\mathfrak{H}}\right) }}{\sqrt{\sum_{i,j=1}^{d}\mathrm{Var}\left(q_j^{-1}\left\langle DF_{n,i},DF_{n,j}\right\rangle _{\mathfrak{H}}\right) }}\rho_{i j k} \mathrm{E}\left[\partial_{i j k} g(Z)\right]:n\geq 1 \right\rbrace ,
 \end{equation}
where the constants $\rho_{i j k}$ are defined by $\rho_{i j k}=\mathrm{E}\left[\tilde{Z}_{i j} Z_k\right]$ whenever \eqref{Condition 3} is true and $\rho_{i j k}=0$ otherwise. If the $\liminf$ and $\limsup$ of \eqref{Condition 4} are not equal to zero and finite, then $$\sqrt{\sum_{i,j=1}^{d}\mathrm{Var}\left(q_j^{-1}\left\langle DF_{n,i},DF_{n,j}\right\rangle _{\mathfrak{H}}\right) }$$ provides an optimal rate of convergence for $F_n$ with respect to the distance defined as 
\begin{equation}\label{Condition 5}
	d(F_n,Z)=\sup\left\lbrace\left|\mathrm{E}\left[ g(F_n) \right] - \mathrm{E}\left[ g(Z) \right]   \right|  \right\rbrace, 
\end{equation}
where $g: \mathbb{R}^d \rightarrow \mathbb{R}$ runs over the class of all real functions that are three-times differentiable with bounded derivatives up to order three. Sufficient conditions analogously to \eqref{Condition 1} and \eqref{Condition 2} for the convergence in law of random sequence \eqref{Sequence} to a centered Gaussian random vector are established in \cite[Proposition 4.2]{Campese2013}. One should note that the techniques developed by Campese in \cite{Campese2013} is extensive and heuristic. In the Campese's framework, smooth functionals of Gaussian processes of which components not necessarily belong to Wiener chaoses are considered and covariances of these smooth functionals are allowed to fluctuate. However, due to the assumption that the random sequence \eqref{Sequence} converges in law, it seems that Campese's findings in \cite{Campese2013} do not offer a complete characterization of the optimal rate of convergence for $F_n$ with respect to the distance defined as \eqref {Condition 5}. For example, Campese provided a counterexample to show that techniques he established can not work if the kernels involved are step functions (see \cite[Section 5.1]{Campese2013} or Section \ref{Section 4.2} in this paper). Note that, in this counterexample, all components of $F_n$ belong to the second Wiener chaos of some isonormal Gaussian process. In addition, Campese remarked in \cite[Section 5.4]{Campese2013} that for a non-trivial application of the results he obtained to Breuer-Major central limit theorem, at least one of the integers $q_i$, the order of $F_{n,i}$, should be even. 
	
In this paper, we consider a sequence of $d$-dimensional random vectors $\left\{F_n=\right. $ $\left.(F_{n,1}, \ldots, F_{n,d}): n \geq 1\right\}$ of which components all belong to $q$-th Wiener chaos, where $q\geq2$. Still suppose that the covariance matrix of $F_n$ is $C$ and $F_n$ converges in distribution to $d$-dimensional normal random vector $Z \sim \mathscr{N}_d(0,C)$. Without any other assumptions, we exhaustively investigate the optimal rate of convergence with respect to the smooth distance $\rho(\cdot,\cdot)$ defined as  
	\begin{equation}
		\rho(F,G)=\sup\left\lbrace\left|\mathrm{E}\left[ g(F) \right] - \mathrm{E}\left[ g(G) \right]   \right|  \right\rbrace,
	\end{equation} 
	where $g:\mathbb{R}^d\rightarrow \mathbb{R}$ runs over the class of all four-times continuously differentiable functions such that $g$ and all of its derivatives of order up to four are bounded by one, and $F$, $G$ are two $d$-dimensional random vectors. Specifically, in Theorem \ref{Main result1}, we get that 
	\begin{equation}
		\rho\left(F_n, Z\right)\asymp M(F_n):= \max\left\lbrace \sum_{|m|=3}\left|  \kappa_{m}(F_n)\right|, \sum_{i=1}^{d}\kappa_{4e_i}(F_n)\right\rbrace,
	\end{equation}
	where for a multi-index $m$, $\kappa_{m}(F_n)$ is the cumulant of order $m$ of $F_n$ (see Definition \ref{Def of cumulant}). That is, the
	concise expression $\max\left\lbrace \sum_{|m|=3}\left|  \kappa_{m}(F_n)\right|, \sum_{i=1}^{d}\kappa_{4e_i}(F_n)\right\rbrace$ is the optimal rate of convergence with respect to the smooth distance $\rho(\cdot,\cdot)$. One can show that $$\sum_{|m|=3}\left|  \kappa_{m}(F_n)\right|\leq c \sqrt{\sum_{i=1}^{d}\kappa_{4e_i}(F_n)},$$ where $c$ is a constant only depending on $q$ and $C$, by combining the interpolation techniques (see \cite[Theorem 4.2]{NP2010} or \cite[Theorem 7.2]{NPR2010}) and \cite[Equation (6.2.6)]{nourdin2012normal}. This is an extension of \cite[Prposition 3.1]{NP2010} to the multidimensional case. Note that $M(F_n)$ can be acquired by either one of the two quantities $\sum_{|m|=3}\left|  \kappa_{m}(F_n)\right|$ and $\sum_{i=1}^{d}\kappa_{4e_i}(F_n)$ (see Section \ref{Section 4} for examples of both cases). Compared to the techniques used in \cite{Campese2013} by Campese, besides Malliavin calculus and Stein's method for normal approximation, we also make full use of method of cumulants. More precisely, in Proposition \ref{expand}, we expand $\mathrm{E}[\left\langle F, \nabla g(F)\right\rangle _{\mathbb{R}^d}]$ as a sum associated with cumulants and related $\Gamma$-random variables by utilizing the formula of integration by parts (see Lemma \ref{IBP}) and the relation between cumulant and related $\Gamma$-random variable (see Theorem \ref{cumulant and Gamma}). On the one hand, combining this expansion and technical estimates of $\Gamma$-random variable (see Proposition \ref{estimation of Gamma}), we prove the upper bound, namely there exists a constant $0<c_2<\infty$ such that for $n$ large enough, $$	\rho\left(F_n, Z\right)\leq c_2 \max\left\lbrace \sum_{|m|=3}\left|  \kappa_{m}(F_n)\right|, \sum_{i=1}^{d}\kappa_{4e_i}(F_n)\right\rbrace.$$ On the other hand, we delicately set up several specific test functions $g$ (see Lemma \ref{specific test function}) to get the lower bound. That is, there exists a constant $0<c_1<\infty$ such that for $n$ large enough, $$	\rho\left(F_n, Z\right)\geq c_1 \max\left\lbrace \sum_{|m|=3}\left|  \kappa_{m}(F_n)\right|, \sum_{i=1}^{d}\kappa_{4e_i}(F_n)\right\rbrace.$$ 
	Note that under the assumption that all components of $F_n$ belong to the same fixed Wiener chaos, the optimal rate of convergence we obtained is comparatively concise. In some degree, this result is consistent with \cite[Theorem 1.5]{NOREDDINE20111008}, in which Noreddine and Nourdin proved that
	\begin{equation}
		\sup\left\lbrace\left|\mathrm{E}\left[ g(F_n) \right] - \mathrm{E}\left[ g(Z) \right]   \right|  \right\rbrace
		\leq c_1 \sum_{i=1}^{d}\sqrt{\kappa_{4e_i}(F_n)}, 
	\end{equation}
	where $g: \mathbb{R}^d \rightarrow \mathbb{R}$ runs over the class of all twice continuously differentiable functions of which second derivatives are bounded by one. We make this assumption since we are not sure whether estimates of cumulants and related $\Gamma$-random variables analogue to Proposition \ref{estimation of Gamma} are still valid for vector-valued Wiener-It\^o integrals of which components allow to belong to Wiener chaoses with different orders. If it is true, we can remove the restriction that all components of vector-valued Wiener-It\^o integrals belong to the same fixed Wiener chaos. This improving topic will be investigated in other works.
	
	 As an application, we first consider a sequence of complex Wiener-It\^o integrals $\left\lbrace F_n: n\geq1\right\rbrace$ in Section \ref{Section 4.1}. Assume that $F_n$ converges in distribution to a complex normal random variable $Z$ with the same covariance matrix as $F_n$. Combining Theorem \ref{Main result1} and the fact that the real and imaginary parts of a complex Wiener-It\^o integral can be expressed as a real Wiener-It\^o integral respectively (see \cite[Theorem 3.3]{chen2017fourth}), we yield Theorem \ref{Main result2}, which states that 
	\begin{equation}
		\rho\left(F_n, Z\right) \asymp \max\left\lbrace \left| \mathrm{E}\left[F_n^3 \right] \right|, \left|\mathrm{E}\left[F_n^2\bar{F_n} \right] \right|, \mathrm{E}\left[\left| F_n \right| ^4 \right]-2\left(\mathrm{E}\left[ \left| F_n \right| ^2\right]  \right) ^2-\left| \mathrm{E}\left[ F_n^2\right] \right| ^2 \right\rbrace.
	\end{equation}
As an example, we get the optimal rate of convergence for a statistic associated with the least squares estimator of the drift coefficient for the complex-valued Ornstein-Uhlenbeck process. In Section \ref{Section 4.2}, we consider the counterexample provided by Campese in \cite[Section 5.1]{Campese2013} and apply our conclusion to derive the optimal rate of convergence for a sequence of vector-valued Wiener-It\^o integrals with kernels of step functions. In Section \ref{Section 4.3}, by combining our techniques and some results from the literature such as \cite{Campese2013,Ginovian1994,Ginovyan2005},  we get the optimal rate of convergence in the multidimensional normal approximation of
vector-valued Toeplitz quadratic functionals.

The paper is organized as follows. Section \ref{Section 2} introduces some elements of the isonormal Gaussian process, Malliavin calculus, method of cumulants and multidimensional Stein's method for normal approximation. In Section \ref{Section 3}, we obtain the optimal rate of convergence for a sequence of vector-valued Wiener-It\^o integrals with respect to smooth distance $\rho(\cdot,\cdot)$. In Section \ref{Section 4}, we apply the main results we proved in Section \ref{Section 3} to get the optimal rates of convergences for a sequence of complex Wiener-It\^o integrals, vector-valued Wiener-It\^o integrals with kernels of step functions and vector-valued Toeplitz quadratic functionals.
	
	\section{Preliminaries}\label{Section 2}
	In this section, we briefly introduce some basic theories of the isonormal Gaussian process, Malliavin calculus, cumulants and multidimensional Stein's method. See \cite{chenliu2019,ito1952complex,nourdin2012normal,nualart2006malliavin} for more details.
	\subsection{Isonormal Gaussian process}
	
	Suppose that $\mathfrak{H}$ is a real separable Hilbert space with an inner product denoted by $\left\langle \cdot, \cdot\right\rangle _\mathfrak{H}$. Let $\|h\|_{\mathfrak{H}}$ denote the norm of $h\in\mathfrak{H}$. Consider a real isonormal Gaussian process $X=\left\{ X(h): h\in\mathfrak{H}\right\} $ defined on a complete probability space $(\Omega, \mathcal{F}, P)$, where the $\sigma$-algebra $\mathcal{F}$ is generated by $X$. That is, $X=\left\{ X(h): h\in\mathfrak{H}\right\} $ is a Gaussian family of centered random variables such that $\mathrm{E}\left[ X(h)X(g) \right]=\left\langle h,g\right\rangle _\mathfrak{H}$ for any $h,g\in \mathfrak{H}$. 
	
	For $q\geq0$, the $q$-th Wiener-It\^{o} chaos $\mathcal{H}_q(X)$ of $X$ is the closed linear subspace of $L^2(\Omega)$ generated by the random variables $\left\{ H_q(X(h)): h\in\mathfrak{H}, \|h\|_{\mathfrak{H}}=1\right\} $, where $H_q(x)$ is the Hermite polynomial of degree $q$ defined by the equality
	\begin{equation*}
		\exp\left\lbrace tx-\frac{1}{2}t^2\right\rbrace =\sum_{q=0}^{\infty}\frac{t^q}{q!}H_q(x).
	\end{equation*}
	Let $\mathfrak{H}^{\otimes q}$ and $\mathfrak{H}^{\odot q}$ denote the $q$-th tensor product and the $q$-th symmetric tensor product of $\mathfrak{H}$, respectively. For any $q \geq 1$, the mapping $I_q\left(h^{\otimes q}\right)=H_q(X(h))$ can be extended to a linear isometry between the symmetric tensor product $\mathfrak{H}^{\odot q}$, equipped with the norm $\sqrt{q!}\|\cdot\|_{\mathfrak{H}^{\otimes q}}$, and the $q$-th Wiener-It\^{o} chaos $\mathcal{H}_{q}(X)$. For $q=0$, we write $I_0(c)=c$ for $c \in \mathbb{R}$. For any $f\in \mathfrak{H}^{\odot q}$, the random variable $I_q(f)$ is called the real $q$-th Wiener-It\^o integral of $f$ with respect to $X$.	Wiener-It\^o chaos decomposition of $L^2(\Omega,\sigma(X),P)$ implies that $L^{2}(\Omega)$ can be decomposed into the infinite orthogonal sum of the spaces $\mathcal{H}_{n}(X)$. That is, any random variable $F \in L^2(\Omega,\sigma(X),P)$ admits a unique expansion of the form
	\begin{equation*}
		F=\sum_{q=0}^{\infty} I_{q}\left(f_{q}\right),	
	\end{equation*}
	where $f_{0}=\mathrm{E}[F]$, and $f_{q} \in \mathfrak{H}^{\odot q}$ with $q \geq1$ are uniquely determined by $F$.
	
	Let $\left\lbrace\eta_k,k\geq1 \right\rbrace $ be a complete orthonormal system in $\mathfrak{H}$. Given $f\in\mathfrak{H}^{\odot p}$, $g\in\mathfrak{H}^{\odot q}$, for $r=0,\dots,p\land q$, the $r$-th contraction of $f$ and $g$ is an element of $\mathfrak{H}^{\otimes (p+q-2r)}$ defined by 
	\begin{equation*}
		f\otimes_rg=\sum_{i_1,\ldots,i_r=1}^{\infty}\left\langle f,\eta_{i_1}\otimes\cdots\otimes \eta_{i_r}\right\rangle _{\mathfrak{H}^{\otimes r}}\otimes\left\langle g,\eta_{i_1}\otimes\cdots\otimes \eta_{i_r}\right\rangle _{\mathfrak{H}^{\otimes r}}.
	\end{equation*}
	Notice that $f\otimes_rg$ is not necessarily symmetric, we denote by $f\tilde{\otimes}_rg$ or $\mathrm{symm}(f\otimes_r g)$ its symmetrization. \cite[Proposition 2.7.10]{nourdin2012normal} provides the product formula for real  multiple Wiener-It\^{o} integrals as follows. For $f\in\mathfrak{H}^{\odot p}$ and $g\in\mathfrak{H}^{\odot q}$ with $p,q\geq0$, 
	\begin{equation}\label{Product_formula}
		I_p(f)I_q(g)=\sum_{r=0}^{p\land q}r!\binom{p}{r}\binom{q}{r}I_{p+q-2r}(f\tilde{\otimes}_rg).
	\end{equation}

	Next, we introduce the complex isonormal Gaussian process. We complexify $\mathfrak{H}$, $L^2(\Omega)$ in the usual way and denote by $\mathfrak{H}_{\mathbb{C}}$, $L^2_{\mathbb{C}}(\Omega)$ respectively. Suppose $\mathfrak{h}=f+\i g\in\mathfrak{H}_{\mathbb{C}}$ with $f,g\in\mathfrak{H}$, we write $X_{\mathbb{C}}(\mathfrak{h}):=X(f)+\i X(g), $ which satisfies $\mathrm{E}\left[X_{\mathbb{C}}\left(\mathfrak{h}\right)\overline{X_{\mathbb{C}}\left(\mathfrak{h}'\right)}\right]=\left\langle \mathfrak{h},\mathfrak{h}'\right\rangle_{\mathfrak{H}_{\mathbb{C}}}$ with $\mathfrak{h}'\in\mathfrak{H}_{\mathbb{C}}$. Let $Y=\left\{ Y(h): h\in\mathfrak{H}\right\} $ is an independent copy of the isonormal Gaussian process $X$ over $\mathfrak{H}$. Define $Y_{\mathbb{C}}(\mathfrak{h})$ same as above.
Let$Z(\mathfrak{h}):=\frac{X_{\mathbb{C}}(\mathfrak{h}) + \i Y_{\mathbb{C}}(\mathfrak{h})}{\sqrt{2}}$ for $ \mathfrak{h}\in \mathfrak{H}_{\mathbb{C}}$,
and we call $Z=\left\{ Z(\mathfrak{h}): \mathfrak{h}\in \mathfrak{H}_{\mathbb{C}}\right\} $  a complex isonormal Gaussian process over $\mathfrak{H}_{\mathbb{C}}$, which is a centered symmetric complex Gaussian family satisfying
\begin{equation}
	\mathrm{E}[Z(\mathfrak{h})^2]=0,\quad
	\mathrm{E}[Z(\mathfrak{h})\overline{Z(\mathfrak{h}')}]=\left\langle \mathfrak{h},\mathfrak{h}' \right\rangle_{\mathfrak{H}_{\mathbb{C}}}, \; \forall \mathfrak{h},\mathfrak{h}'\in \mathfrak{H}_{\mathbb{C}}.
\end{equation}

For each $p,q\geq 0$, let $\mathscr{H}_{p,q}(Z)$ be the closed linear subspace of $L^2_{\mathbb{C}}(\Omega)$ generated by the random variables $\left\{ J_{p,q}(Z(\mathfrak{h})): \mathfrak{h}\in\mathfrak{H}_{\mathbb{C}},\|\mathfrak{h}\|_{\mathfrak{H}_{\mathbb{C}}}=\sqrt2\right\} $, where $J_{p,q}(z)$ is the complex Hermite polynomial, or Hermite-Laguerre-It\^o polynomial, given by 
\begin{equation*}
	\exp\left\{ \lambda\bar{z}+\bar{\lambda}z-2|\lambda|^2\right\} =\sum_{p=0}^{\infty}\sum_{q=0}^{\infty}\frac{\bar{\lambda}^p\lambda^q}{p!q!}J_{p,q}(z),\;\lambda\in\mathbb{C}.
\end{equation*}
The space $\mathscr{H}_{p,q}(Z)$ is called the $(p, q)$-th Wiener-It\^o chaos of $Z$.

Take a complete orthonormal system $\left\{ \xi_k, k\geq1\right\} $ in $\mathfrak{H}_{\mathbb{C}}$. We denote by $\Lambda$ the set of all sequences $\textbf{a}=\left\lbrace a_k \right\rbrace_{k=1}^{\infty}$ of non-negative integers with only finitely many nonzero components. For two sequences $\textbf{p}=\left\{ p_k \right\}_{k=1}^{\infty}, \textbf{q}=\left\{ q_k \right\}_{k=1}^{\infty}\in\Lambda$, the linear mapping
\begin{equation}\label{def of complex integral}
	{\I}_{p,q}\left(\rm{symm}\left(\otimes_{k=1}^{\infty}{\xi}_k^{\otimes p_k}\right)\otimes \rm{symm}\left(\otimes_{k=1}^{\infty}\overline{\xi}_k^{\otimes q_k}\right)\right):= \prod_{k=1}^{\infty}\frac{1}{\sqrt{2^{p_k+q_k}}}J_{p_k,q_k}\left(\sqrt2Z\left({\xi}_k\right)\right),
\end{equation}
provides an isometry from the tensor product $\mathfrak{H}_{\mathbb{C}}^{\odot p}\otimes\mathfrak{H}_{\mathbb{C}}^{\odot q}$, equipped with the norm $\sqrt{p!q!}\|\cdot\|_{\mathfrak{H}_{\mathbb{C}}^{\otimes (p+q)}}$, onto the $(p,q)$-th Wiener-It\^o chaos $\mathscr{H}_{p,q}(Z)$.  Note that \eqref{def of complex integral} was proved by It\^o in \cite[Theorem 13.2]{ito1952complex}. For any $f\in\mathfrak{H}_{\mathbb{C}}^{\odot p}\otimes\mathfrak{H}_{\mathbb{C}}^{\odot q}$, $\I_{p,q}(f)$ is called complex $(p,q)$-th Wiener-It\^o integral of $f$ with respect to $Z$. Complex Wiener-It\^o chaos decomposition of $L_{\mathbb{C}}^2(\Omega,\sigma(Z),P)$ implies that $L_{\mathbb{C}}^2(\Omega,\sigma(Z),P)$ can be decomposed into the infinite orthogonal sum of the spaces $\mathscr{H}_{p,q}(Z)$. That is, any random variable $F \in L_{\mathbb{C}}^2(\Omega,\sigma(Z),P)$ admits a unique expansion of the form
\begin{equation}\label{complex chaos decomposition}
	F=\sum_{p=0}^{\infty}\sum_{q=0}^{\infty} \I_{p,q}\left(f_{p,q}\right),	
\end{equation}
where $f_{0,0}=\mathrm{E}[F]$, and $f_{p,q} \in \mathfrak{H}_{\mathbb{C}}^{\odot p}\otimes\mathfrak{H}_{\mathbb{C}}^{\odot q}$ with $p+q \geq1$, are uniquely determined by $F$.

Given $f\in\mathfrak{H}_{\mathbb{C}}^{\odot a}\otimes\mathfrak{H}_{\mathbb{C}}^{\odot b}$, $g\in\mathfrak{H}_{\mathbb{C}}^{\odot c}\otimes\mathfrak{H}_{\mathbb{C}}^{\odot d}$, for $i=0,\dots,a\land d$, $j=0,\dots,b\land c$, the $(i,j)$-th contraction of $f$ and $g$ is an element of $\mathfrak{H}_{\mathbb{C}}^{\odot (a+c-i-j)}\otimes\mathfrak{H}_{\mathbb{C}}^{\odot (b+d-i-j)}$ defined by
\begin{align}
	f \otimes_{i, j} g
	&= \sum_{l_{1}, \ldots, l_{i+j}=1}^{\infty}\left\langle f, \xi_{l_{1}} \otimes \cdots \otimes \xi_{l_{i}} \otimes \bar{\xi}_{l_{i+1}} \otimes \cdots \otimes \bar{\xi}_{l_{i+j}}\right\rangle\\ &\qquad\qquad\qquad\otimes\left\langle g, \xi_{l_{i+1}} \otimes \cdots \otimes \xi_{l_{i+j}} \otimes \bar{\xi}_{l_{1}} \otimes \cdots \otimes \bar{\xi}_{l_{i}}\right\rangle,
\end{align}
and by convention, $f \otimes_{0,0} g=f \otimes g$ denotes the tensor product of $f$ and $g$. \cite[Theorem 2.1]{Chen2017} and  \cite[Theorem A.1]{Hoshino2017} establish the product formula for complex Wiener-It\^o integrals. For $f \in \mathfrak{H}^{\odot a} \otimes\mathfrak{H}^{\odot b}$ and $ g \in \mathfrak{H}^{\odot c} \otimes \mathfrak{H}^{\odot d}$ with $a, b, c, d \geq0$,
\begin{equation}\label{complex product}
	\I_{a, b}(f) \I_{c, d}(g)=\sum_{i=0}^{a \wedge d} \sum_{j=0}^{b \wedge c}\binom{a}{i} \binom{d}{i}\binom{b}{j}\binom{c}{j} i! j! \I_{a+c-i-j, b+d-i-j}\left(f \otimes_{i, j} g\right).
\end{equation}
	
	\subsection{Malliavin calculus}
	
	Let $\mathcal{S}$ denote the class of smooth random variables of the form $F=f(X(h_1),\dots,X(h_n))$, where $h_1,\dots,h_n\in\mathfrak{H}$, $n\geq1$ and $f\in C_p^{\infty}(\mathbb{R}^n)$, the set of all infinitely continuously differentiable real-valued functions such that $f$ and all of its partial derivatives have polynomial growth. Given $F\in \mathcal{S}$, the Malliavin derivative $DF$ is a $\mathfrak{H}$-valued random element given by 
	\begin{equation*}
		DF=\sum_{i=1}^{n}\frac{\partial f}{\partial x_i}\left(X\left(h_1\right),\dots,X\left(h_n\right)\right)h_i.
	\end{equation*}
	The derivative operator $D$ is a closable and unbounded operator from $L^p(\Omega)$ to $L^p(\Omega;\mathfrak{H})$ for any $p\geq1$. By iteration, for $k\geq2$, one can define $k$-th derivative $D^kF\in L^p(\Omega;\mathfrak{H}^{\otimes k})$. For any $p\geq1$ and $k\geq0$, let $\mathbb{D}^{k,p}$ denote the closure of $\mathcal{S}$ with respect to the norm $\|\cdot\|_{k,p}$ given by 
	\begin{equation*}
		\|F\|_{k,p}^{p}=\sum_{i=0}^{k}\mathrm{E}\left(\left\|D^iF\right\|^{p}_{\mathfrak{H}^{\otimes i}}\right).
	\end{equation*} 
	For any $p\geq1$ and $k\geq0$, we set $\mathbb{D}^{\infty,p}=\bigcap_{k\geq0}\mathbb{D}^{k,p}$, $\mathbb{D}^{k,\infty}=\bigcap_{p\geq1}\mathbb{D}^{k,p}$ and $\mathbb{D}^{\infty}=\bigcap_{k\geq0}\mathbb{D}^{k,\infty}$. If $F=I_p(f)$ with $f\in\mathfrak{H}^{\odot p}$, then $I_p(f)\in\mathbb{D}^{\infty}$ and for any $k\geq0$,
	\begin{equation*}
		D^kI_p(f) =
		\begin{cases}
			\frac{p!}{(p-k)!}I_{p-k}(f),  & {k\leq p,}\\
			0,  & {k>p.}
		\end{cases}
	\end{equation*} 	
	The derivative operator $D$ satisfies the chain rule. Specifically, if $\varphi: \mathbb{R}^n \rightarrow \mathbb{R}$ is continuously differentiable with bounded partial derivatives and $F=\left(F_1, \ldots, F_n\right)$ is a vector of elements of $\mathbb{D}^{1,2}$, then $\varphi(F) \in \mathbb{D}^{1,2}$ and
	$$
	D \varphi(F)=\sum_{i=1}^n \frac{\partial \varphi}{\partial x_i}(F) D F_i .
	$$
	The chain rule still holds if $F_i \in \mathbb{D}^{\infty}$ and $\varphi$ has continuous partial derivatives with at most polynomial growth.
	
	We denote by $\delta$ the divergence operator, defined as the adjoint operator of $D$, which is an unbounded operator from a domain in $L^2(\Omega;\mathfrak{H})$ to $L^2(\Omega)$. A random element $u\in L^2(\Omega;\mathfrak{H})$ belongs to the domain of $\delta$, denoted Dom$\delta$, if and only if it verifies 
	\begin{equation*}
		\left|\mathrm{E}\left[\langle D F, u\rangle_{\mathfrak{H}}\right]\right| \leq c_{u} \sqrt{\mathrm{E}\left[F^{2}\right]},
	\end{equation*}
	for any $F \in \mathbb{D}^{1,2}$, where $c_{u}$ is a constant depending only on $u$. In particular, if $u \in \mathrm{Dom} \delta$, then $\delta(u)$ is characterized by the following duality relationship
	\begin{equation}\label{integration by parts formula}
		\mathrm{E}( F\delta(u))=\mathrm{E}\left(\langle D F, u\rangle_{\mathfrak{H})}\right),
	\end{equation}
	for any $F \in \mathbb{D}^{1,2}$. 
	
	The operator $L$ defined as $L=-\sum_{n=0}^{\infty}n J_{n}$ is the infinitesimal generator of the Ornstein-Uhlenbeck semigroup $T_{t}=\sum_{n=0}^{\infty} e^{-n t} J_{n}$. Its domain in $L^{2}(\Omega)$ is
	$$
	\mathrm{Dom} L=\left\{F \in L^{2}(\Omega): \sum_{n=1}^{\infty} n^{2}\left\|J_{n} F\right\|_{2}^{2}<\infty\right\}.
	$$
	\cite[Proposition 1.4.3]{nualart2006malliavin} states the relation between the operators $D$, $\delta$ and $L$. For $F \in L^{2}(\Omega)$, $F \in \mathrm{Dom} L$ if and only if $F \in \mathrm{Dom}(\delta D)$, and in this case
	\begin{equation}\label{delta DL}
		\delta D F=-L F.
	\end{equation}
	For any $F \in L^{2}(\Omega)$, we also define $L^{-1} F=-\sum_{n=1}^{\infty}\frac{1}{n} J_{n}(F)$. The operator $L^{-1}$ is called the pseudo-inverse of $L$. For any $F \in L^{2}(\Omega)$, we have that $L^{-1} F \in \mathrm{Dom} L$, and
	\begin{equation}\label{LL-1}
		LL^{-1}F=L^{-1}LF=F-\mathrm{E}[F].
	\end{equation}
	Combining \eqref{integration by parts formula}, \eqref{delta DL} and \eqref{LL-1}, we can get the following useful lemma.
	\begin{Lemma}{\cite[Lemma 2.1]{NOREDDINE20111008}}\label{IBP}
		Suppose that $F \in \mathbb{D}^{1,2}$ and $G \in L^2(\Omega)$. Then, $L^{-1} G \in \mathbb{D}^{2,2}$ and 
		\begin{equation}
			\mathrm{E}[F G]=\mathrm{E}[F] \mathrm{E}[G]+\mathrm{E}\left[\left\langle D F,-D L^{-1} G\right\rangle_{\mathfrak{H}}\right] .
		\end{equation}
	\end{Lemma}

	\subsection{Cumulants}
	
	First, we recall some standard multi-index notations. A multi-index is defined as a $d$-dimensional vector $m=\left(m_1, \ldots, m_d\right)\in\mathbb{N}_0^d=\left( \mathbb{N} \cup\{0\}\right) ^d$. For ease of notations, we write $|m|=\sum_{i=1}^d m_i$, $\partial_i=\frac{\partial}{\partial x_i}$, $\partial^m=\partial_1^{m_1} \ldots \partial_d^{m_d}$ and $x^m=\prod_{i=1}^d x_i^{m_i}$. By convention, we have $0^0=1$. For any $i=1, \ldots, d$, we denote by $e_i \in \mathbb{N}_0^d$ the multi-index of order one defined by $\left(e_i\right)_j=\delta_{i j}$, where $\delta_{i j}$ the Kronecker symbol. We can write every multi-index $m$ as a sum of $|m|$ multi-indices $l_1, \ldots, l_{|m|}$ of order one, and this sum is unique up to the order of the summands. For instance, the elementary decomposition for the multi-index $(1,2,0)$ is $\{(1,0,0),(0,1,0),(0,1,0)\}$.
	
	\begin{Def}\label{Def of cumulant}
		Let $F=\left(F_1, \ldots, F_d\right)$ be a $d$-dimensional random vector satisfying $\mathrm{E}|F|^m<\infty$ for some $m \in \mathbb{N}_0^d \backslash\{0\}$. The characteristic function of $F$ is denoted by $\phi_F(t)=\mathrm{E}\left[\mathrm{e}^{\mathrm{i}\langle t, F\rangle_{\mathbb{R}^d}}\right]$ for $t \in \mathbb{R}^d$. Then the cumulant of order $m$ of $F$ is defined as
		$$
		\kappa_m(F)=\left.(-\mathrm{i})^{|m|} \partial^m \log \phi_F(t)\right|_{t=0} .
		$$
		For example, if $F_i, F_j \in L^2(\Omega)$, then $\kappa_{e_i}(F)=\mathrm{E}\left[F_i\right]$ and $\kappa_{e_i+e_j}(F)=\operatorname{Cov}\left(F_i, F_j\right)$.
	\end{Def}
	
	\begin{Def}
		Let $F=\left(F_1, \ldots, F_d\right)$ be a $d$-dimensional random vector with $F_i \in \mathbb{D}^{1,2}$ for $1\leq i\leq d$. Suppose that  $l_1, l_2, \ldots$ is a sequence taking values in $\left\{e_1, \ldots, e_d\right\}$. Set $\Gamma_{l_1}(F)=F^{l_1}$. If the random variable $\Gamma_{l_1, \ldots, l_k}(F)$ is a well-defined element of $L^2(\Omega)$ for some $k \geq 1$, we set
		$$
		\Gamma_{l_1, \ldots, l_{k+1}}(F)=\left\langle D F^{l_{k+1}},-D L^{-1} \Gamma_{l_1, \ldots, l_k}(F)\right\rangle_{\mathfrak{H}} .
		$$
	\end{Def}
    \cite[Lemma 4.3]{NOREDDINE20111008} shows that 
	if $F_i\in \mathbb{D}^{\infty}$ for $1\leq i\leq d$, then for any $k\geq1$, the random variable $\Gamma_{l_1, \ldots, l_{k}}(F)\in\mathbb{D}^{\infty}$. The following theorem tells us the relation between cumulant $\kappa_{m}(F)$ and the random variable $\Gamma_{l_1, \ldots, l_{|m|}}(F)$ with $m=l_1+ \ldots+ l_{|m|}$.
	\begin{Thm}{\cite[Theorem 4.4]{NOREDDINE20111008}}\label{cumulant and Gamma}	
		Let $m \in \mathbb{N}_0^d \backslash\{0\}$. Write $m=l_1+\cdots+l_{|m|}$ where $l_i \in\left\{e_1, \ldots, e_d\right\}$ for each $i$. Suppose that the random vector $F=\left(F_1, \ldots, F_d\right)$ is such that $F_i \in \mathbb{D}^{|m|, 2^{|m|}}$ for all $i$. Then, we have
		$$
		\kappa_m(F)=(|m|-1) ! \mathrm{E}\left[\Gamma_{l_1, \ldots, l_{|m|}}(F)\right] .
		$$
	\end{Thm}
	If the components of random vector $F=\left(F_1, \ldots, F_d\right)$ are all multiple Wiener-It\^o integrals, namely, $F=\left(I_{q_1}\left(f_1\right), \ldots, I_{q_d}\left(f_d\right)\right)$, where each $f_i$ belongs to $\mathfrak{H}^{\odot q_i}$, the cumulant $\kappa_{m}(F)$ and the random variable $\Gamma_{l_1, \ldots, l_{|m|}}(F)$ with $m=l_1+ \ldots+ l_{|m|}$ and $l_i \in\left\{e_1, \ldots, e_d\right\}$ can be expressed more clearly, see \cite[Theorem 4.6, Equation (4.29)]{NOREDDINE20111008}. 
	Specifically, set $\lambda_k=j$ when $l_k=e_j$.  For simplicity, we drop the brackets and write $f_{\lambda_1} \tilde{\otimes}_{r_2} \cdots \tilde{\otimes}_{r_{|m|-1}} f_{\lambda_{|m|-1}}$ to implicitly assume that this quantity is defined iteratively from the left to the right. For instance, $f \tilde{\otimes}_\alpha g \tilde{\otimes}_\beta h $ actually means $\left(f \tilde{\otimes}_\alpha g\right) \tilde{\otimes}_\beta h$. Then
	\begin{align}\label{Gamma of integral}
		&\Gamma_{l_1, \ldots, l_{|m|}}(F)\\= &\,  \sum_{r_2=1}^{q_{\lambda_1} \wedge q_{\lambda_2}} \ldots \sum_{r_{|m|}=1}^{[q_{\lambda_1}+\cdots+q_{\lambda_{|m|-1}}-2 r_2-\cdots-2 r_{|m|-1}] \wedge q_{\lambda_{|m|}}} c_{q, l}\left(r_2, \ldots, r_{|m|}\right) \mathbf{1}_{\left\{r_2<\frac{q_{\lambda_1}+q_{\lambda_2}}{2}\right\}}\times \ldots  \\
		&\, \times \mathbf{1}_{\left\{r_2+\cdots+r_{|m|-1}<\frac{q_{\lambda_1}+\cdots+q_{\lambda_{|m|-1}}}{2}\right\}} I_{q_{\lambda_1}+\cdots+q_{\lambda_{|m|}}-2 r_2-\cdots-2 r_{|m|}}\left(f_{\lambda_1} \tilde{\otimes}_{r_2} f_{\lambda_2} \ldots \tilde{\otimes}_{r_{|m|}} f_{\lambda_{|m|}}\right) .
	\end{align}
	and
	\begin{equation}\label{Kappa of integral}
		\kappa_m(F)=(q_{\lambda_{|m|}}) !(|m|-1) ! \sum c_{q, l}\left(r_2, \ldots, r_{|m|-1}\right)\left\langle f_{\lambda_1} \tilde{\otimes}_{r_2} f_{\lambda_2} \cdots \tilde{\otimes}_{r_{|m|-1}} f_{\lambda_{|m|-1}}, f_{\lambda_{|m|}}\right\rangle_{\mathfrak{H}^{\otimes q_{\lambda_{|m|}}}} ,
	\end{equation}
	where the sum $\sum$ runs over all collections of integers $r_2, \ldots, r_{|m|-1}$ such that
	\begin{enumerate}[(i)]
		\item $1 \leq r_i \leq q_{\lambda_i}$ for all $i=2, \ldots,|m|-1$;
		\item $r_2+\cdots+r_{|m|-1}=\frac{q_{\lambda_1}+\cdots+q_{\lambda_{|m|-1}}-q_{\lambda_{|m|}}}{2}$;
		\item $r_2<\frac{q_{\lambda_1}+q_{\lambda_2}}{2}, \ldots, r_2+\cdots+r_{|m|-2}<\frac{q_{\lambda_1}+\cdots+q_{\lambda_{|m|-2}}}{2}$;
		\item $r_3 \leq q_{\lambda_1}+q_{\lambda_2}-2 r_2, \ldots, r_{|m|-1} \leq q_{\lambda_1}+\cdots+q_{\lambda_{|m|-2}}-2 r_2-\cdots-2 r_{|m|-2}$;
	\end{enumerate}
	and where the combinatorial constants $c_{q, l}\left(r_2, \ldots, r_s\right)$ are recursively defined by the relations
	\begin{equation}
		c_{q,l}\left(r_2\right)=q_{\lambda_2}\left(r_2-1\right) !\binom{q_{\lambda_1}-1}{r_2-1}\binom{	q_{\lambda_2}-1}{r_2-1},
	\end{equation}
	and, for $s \geq 3$,
	\begin{align}\label{c_q,l}
		c_{q, l}\left(r_2, \ldots, r_s\right)&=q_{\lambda_s}\left(r_s-1\right) !\binom{	q_{\lambda_1}+\cdots+q_{\lambda_{s-1}}-2 r_2-\cdots-2 r_{s-1}-1}{r_s-1}\\&\quad\binom{q_{\lambda_s}-1}{r_s-1} c_{q, l}\left(r_2, \ldots, r_{s-1}\right).
	\end{align} 
	In particular, if $q_1= \cdots=q_d=2$, then the only possible integers $r_2, \ldots, r_{|m|-1}$ satisfying (i)-(iv) are $r_2=\cdots=r_{|m|-2}=1$. Computing directly, one can get that $c_{q, l}(1)=2, c_{q, l}(1,1)=4, c_{q, l}(1,1,1)=8$, and so on. Therefore, for any $f_1, \ldots, f_d \in \mathfrak{H}^{\odot 2}$ and any $m \in \mathbb{N}_0^d \backslash\{0\}$ with $|m| \geq 3$, we have
	\begin{equation}\label{cumulant for second chaos}
			\kappa_m\left(I_2\left(f_1\right), \ldots, I_2\left(f_d\right)\right)=2^{|m|-1}(|m|-1) !\left\langle f_{\lambda_1} \tilde{\otimes}_1 \cdots \tilde{\otimes}_1 f_{\lambda_{|m|-1}}, f_{\lambda_{|m|}}\right\rangle_{\xi^{\otimes 2}}.
	\end{equation}

	\subsection{Multidimensional Stein's method for normal approximations}

	In this subsection, we denote by
	$\mathcal{M}_d (\mathbb{R})$ the collection of all real $d \times d$ matrices. The Hilbert-Schmidt inner product and the Hilbert-Schmidt norm on $\mathcal{M}_{d}(\mathbb{R})$, denoted respectively by $\langle\cdot, \cdot\rangle_{\mathrm{HS}}$ and $\|\cdot\|_{\mathrm{HS}}$, are defined as
	\begin{equation*}
		\langle A, B\rangle_{\mathrm{HS}}=\operatorname{Tr}\left(A B^{T}\right), \quad \|A\|_{\mathrm{HS}}=\sqrt{\langle A, A\rangle_{\mathrm{HS}}}, \quad A, B\in \mathcal{M}_d (\mathbb{R}),
	\end{equation*}
	where $\operatorname{Tr}(\cdot)$ and ${(\cdot) }^{T}$ denote the usual trace and transposition operators, respectively.
	
	Let $C=\left( C_{ij}\right) _ {1\leq i,j\leq d}\in\mathcal{M}_d(\mathbb{R})$ be a non-negative definite and symmetric matrix. We denote by $\mathcal{N}_d (0,C)$ the law of an $d$-dimensional Gaussian  vector with zero mean and covariance matrix $C$. Multidimensional Stein's lemma (see \cite[Lemma 4.1.3]{nourdin2012normal}) shows that a random vector $N=\left(N_{1}, \ldots, N_{d}\right)\sim \mathcal{N}_d (0,C)$ if and only if   
	\begin{equation}\label{Stein's lemma}
		\mathrm{E}\left[\langle N, \nabla f(N)\rangle_{\mathbb{R}^{d}}\right]=\mathrm{E}\left[\langle C, \mathrm{Hess} f(N)\rangle_{\mathrm{HS}}\right],
	\end{equation}
	for every $\mathcal{C}^{2}$ function $f: \mathbb{R}^{d} \rightarrow \mathbb{R}$ having bounded first and second derivatives, where $\mathcal{C}^{2}$ is the class of functions having continuous second derivative. Here Hess$f$ denotes the Hessian of $f$, that is, the $d \times d$ matrix of which entries are given by $(\mathrm{ Hess}f)_{ij}=\partial^{2}_{ij}f$.
	
	Suppose $F$ is a $d$-dimensional random vector such that the expectation
	$$\mathrm{E}\left[\langle F, \nabla f(F)\rangle_{\mathbb{R}^{d}}-\langle C, \mathrm{Hess} f(F)\rangle_{\mathrm{HS}}\right],$$ is close to zero for a large class of smooth functions $f$. In view of Stein's Lemma, it is possible to conclude that the law of $F$ is close to the law of $N$. In order to give a quantitative version of Stein's lemma, we introduce the definition of Stein's equation. 
	Suppose the random vector $Z \sim \mathscr{N}_{d}(0, C)$. Let $g: \mathbb{R}^{d} \rightarrow \mathbb{R}$ be such that $\mathrm{E}|g(Z)|<\infty$. The Stein's equation associated with $g$ and $Z$ is the partial differential equation
	\begin{equation}\label{Stein's equation}
		\langle C, \mathrm {Hess} f(x)\rangle_{\mathrm{HS}}-\langle x, \nabla f(x)\rangle_{\mathbb{R}^{d}}=g(x)-\mathrm{E}[g(Z)].
	\end{equation}
	A solution to the Equation \eqref{Stein's equation} is a function $f\in\mathcal{C}^{2}$ satisfying \eqref{Stein's equation} for every $x \in \mathbb{R}^{d}$.
	
	Given a Lipschitz function $g: \mathbb{R}^d \rightarrow \mathbb{R}$ with at most polynomial growth, we define $U_{g, C}: \mathbb{R}^d \rightarrow \mathbb{R}$ by
	\begin{equation}\label{solution of Stein}
		U_{g, C}(x)=\int_a^b \frac{v^{\prime}(t)}{v(t)}\left(\mathrm{E}[g(N)]-\mathrm{E}\left[g\left(v(t) x+\sqrt{1-v^2(t)} N\right)\right]\right) \mathrm{d} t,
	\end{equation}
	where $N\sim \mathscr{N}_{d}(0, C)$ is independent of $Z$, $-\infty \leq a<b \leq \infty$ and $v:(a, b) \rightarrow(0,1)$ is a diffeomorphism with $\lim _{t \rightarrow a+} v(t)=0$ and $\lim _{t \rightarrow b-} v(t)=1$. From the change of variable $v(t)=s$, we see that $U_{g, C}$ does not depend on the choice of $v$. \cite[Lemma 2.4]{Campese2013} shows that $U_{g, C}$ defined as \eqref{solution of Stein} satisfies the multidimensional Stein's equation \eqref{Stein's equation}. Moreover, if $g$ is $k$-times differentiable with bounded derivatives up to order $k$, the same is true for $U_{g, C}$. In this case, for any $m \in \mathbb{N}_0^d$ with $|m| \leq k$, the derivatives are given by
	\begin{equation}\label{Derivative of U}
		\partial^m U_{g, C}(x)=\int_a^b v^{\prime}(t) v^{|m|-1}(t) \mathrm{E}\left[\partial^m g\left(v(t) x+\sqrt{1-v^2(t)} N\right)\right] \mathrm{d} t,
	\end{equation}
	and it holds that
	\begin{equation}\label{Bound of derivative of U}
		\left|\partial^m U_{g, C}(x)\right| \leq \frac{\left\|\partial^m g\right\|_{\infty}}{|m|},
	\end{equation}
	and
	\begin{equation}\label{Expectation of derivative of U}
		\mathrm{E}\left[\partial^m U_{g, C}(Z)\right]=\frac{1}{|m|} \mathrm{E}\left[\partial^m g(Z)\right].
	\end{equation}

	\section{Optimal rate of convergence for  vector-valued Wiener-It\^o integral}\label{Section 3}

	Let $\left\{F_n= (F_{n,1},\ldots,F_{n,d}): n \geq 1\right\}$ be a sequence of random vectors of which all components live in the $q$-th Wiener chaos and $q\geq2$. Suppose that $F_n$ converges in distribution to a $d$-dimensional normal random vector $Z$. Let 
	\begin{equation}\label{optimal rate of convergence}
		M(F_n)=\max\left\lbrace \sum_{|m|=3}\left|  \kappa_{m}(F_n)\right|, \sum_{i=1}^{d}\kappa_{4e_i}(F_n)\right\rbrace.
	\end{equation}
    Note that $M(F_n)\geq \sum_{i=1}^{d}\kappa_{4e_i}(F_n)>0$ (see \cite[Lemma 5.2.4]{nourdin2012normal}) and $M(F_n)\rightarrow M(Z)=0$ under the assumption that that $F_n$ converges in distribution to $Z$.
    
	Define the distance between the distributions of two $d$-dimensional random vectors as 
	\begin{equation}\label{distance}
		\rho(F,G)=\sup\left\lbrace\left|\mathrm{E}\left[ g(F) \right] - \mathrm{E}\left[ g(G) \right]   \right|  \right\rbrace,
	\end{equation} 
	where $g:\mathbb{R}^d\rightarrow \mathbb{R}$ runs over the class of all four-times continuously differentiable functions such that $g$ and all of its derivatives of order up to four are bounded by one.

	\begin{Thm}\label{Main result1}
		Fix $q \geq 2$. Let $\left\{F_n=(F_{n,1},\ldots,F_{n,d}): n \geq 1\right\}$ be a sequence of random vectors of which components live in the $q$-th Wiener chaos. Suppose that the covariance matrix of $F_n$ is $C$ and $F_n$ converges in distribution to $Z \sim \mathscr{N}_d(0,C)$. Then there exist two finite constants $0<c_1<c_2$ only depending on $q$ and $d$ such that for $n$ large enough,
		\begin{equation}
			c_1 M\left(F_n\right) \leq \rho\left(F_n, Z\right) \leq c_2 M\left(F_n\right).
		\end{equation}
	\end{Thm}
    \begin{Rem}
    	From the proof of Theorem \ref{Main result1}, one can get that the upper bound, namely $\rho\left(F_n, Z\right) \leq c_2 M\left(F_n\right)$, is still hold without the assumption that $F_n$ converges in distribution to $Z \sim \mathscr{N}_d(0,C)$.
    \end{Rem}
		\begin{Rem}
		There are two reasons why we consider $M(F_n)$ as the optimal rate of convergence and require the smoothness of test function $g$ in \eqref{distance} to be of order 4. Firstly, combining Proposition \ref{expand} below and Stein's method, the test function $g$ should be at least continuously differentiable up to order three. However, if we take $M=3$ in Proposition \ref{expand}, the remainder term $$\sum_{m=e_{i}+e_{j}+e_{k}, 1\leq i,j,k\leq d} \mathrm{E}\left[\Gamma_{e_{i},e_{j}, e_{k}}(F) \partial^{m}f(F)\right]$$
		is bounded by $\max\left\lbrace \sum_{|m|=3}\left|  \kappa_{m}(F_n)\right|, \sum_{i=1}^{d}\kappa_{4e_i}(F_n)^{\frac{3}{4}}\right\rbrace$ according to Equation \eqref{Gamma3}. The convergence rate of this bound is slower than
		$M(F_n)$, the upper bound we get in Theorem \ref{Main result1} by taking $M=4$ in Proposition \ref{expand}. Secondly, if $M\geq 5$, the reminder term $$\sum_{s=3}^{M-1}\sum_{m=e_{j_1}+\dots+e_{j_s}, \atop 1\leq j_k\leq d, 1\leq k\leq s} \frac{\kappa_{m}(F)}{(s-1)!} \mathrm{E}\left[\partial^{m}f(F)\right]+\sum_{m=e_{j_1}+\dots+e_{j_M},\atop 1\leq j_k\leq d, 1\leq k\leq M} \mathrm{E}\left[\Gamma_{e_{j_1},\ldots, e_{j_M}}(F) \partial^{m}f(F)\right]$$ is still bounded by $M(F_n)$. For example, taking $M=5$, the reminder term
		\begin{align}
			&\sum_{m=e_{i}+e_{j}+e_{k},\atop 1\leq i,j,k\leq d} \frac{\kappa_{m}(F)}{2} \mathrm{E}\left[\partial^{m}f(F)\right]
			+ \sum_{m=e_{i}+e_{j}+e_{k}+e_{l}, \atop1\leq i,j,k,l\leq d} \frac{\kappa_{m}(F)}{3!} \mathrm{E}\left[\partial^{m}f(F)\right]\\+&\,
			\sum_{m=e_{j_1}+\dots+e_{j_5},\atop 1\leq j_k\leq d} \mathrm{E}\left[\Gamma_{e_{j_1},\ldots, e_{j_5}}(F) \partial^{m}f(F)\right]
		\end{align}  
		is bounded by $$\max\left\lbrace \sum_{|m|=3}\left|  \kappa_{m}(F_n)\right|, \sum_{i=1}^{d}\kappa_{4e_i}(F_n), \sum_{i=1}^{d}\kappa_{4e_i}(F_n)^{\frac{5}{4}}\right\rbrace=M(F_n)$$ according to Proposition \ref{estimation of Gamma}. The above two points are the reasons why we define the optimal rate of convergence as \eqref{optimal rate of convergence} and the distance as \eqref{distance}.
	\end{Rem}

In Theorem \ref{Main result1}, we consider the sequences of vector-valued Wiener-It\^o integrals $\left\lbrace F_n:n\geq1 \right\rbrace $ with deterministic covariance matrix $C$. Actually, the conclusion can be extended to the case that the covariance matrix of $F_n$, denoted by $C_n$, converges to $C$ in the meaning of $\|C_n-C\|_{\mathrm{HS}}\rightarrow 0$ as $n\rightarrow\infty$. We introduce the definition of asymptotically close to normal. We say that $\left\lbrace F_n:n\geq1 \right\rbrace $ is asymptotically close to normal if $\rho(F_n,Z_n)\rightarrow0$, where $Z_n$ is a $d$-dimensional Gaussian random vector of which covariance matrix coincides with that of $F_n$. The definition of asymptotically close to normal was introduced in \cite[Definition 2.3]{Campese2013} by Campese with respect to the Prokhorov distance $\beta$, which is equivalent to convergence in law  in the meaning of that $\beta(F_n,Z)\rightarrow0\Leftrightarrow F_n\overset{d}{\rightarrow}Z$, as $n\rightarrow\infty$. Here, we adopt distance $\rho(\cdot,\cdot)$ (see Definition \ref{distance}) which is also  equivalent to convergence in law. Note that if the test function $g$ in the definition of the distance $\rho(\cdot,\cdot)$ is not necessarily bounded, then the topology induced by $\rho(\cdot,\cdot)$ is stronger than the topology of the convergence in distribution. Using the similar argument as in the proof of Theorem \ref{Main result1}, we can obtain the following proposition.
\begin{Prop}
	Fix $q \geq 2$. Let $\left\{F_n=(F_{n,1},\ldots,F_{n,d}): n \geq 1\right\}$ be a sequence of random vectors of which components live in the $q$-th Wiener chaos. Suppose that $C_n$, the covariance matrix of $F_n$, converges to $C$ in the meaning of $\|C_n-C\|_{\mathrm{HS}}\rightarrow 0$ as $n\rightarrow\infty$. 
	\begin{enumerate}[(i)]
		\item If $C$ is invertible, we set $F_n^{'}=C^{\frac{1}{2}}C_n^{-\frac{1}{2}}F_n$ and assume that $F_n^{'}$ converges in distribution to $Z \sim \mathscr{N}_d(0,C)$. Then for $n$ large enough,
		\begin{equation}
			\rho(F_n^{'}, Z)\asymp M(F_n^{'}).
		\end{equation}
	\item If $C$ is not invertible, suppose that $\left\lbrace F_n:n\geq1 \right\rbrace $ is asymptotically close to normal. That is, $\rho(F_n,Z_n)\rightarrow0$, where $Z_n$ is a $d$-dimensional Gaussian random vector of which covariance matrix coincides with that of $F_n$. Then for $n$ large enough,
		\begin{equation}
		    \rho(F_n,Z_n)\asymp M(F_n).
	    \end{equation}
	\end{enumerate}
\end{Prop}

	To prove Theorem \ref{Main result1}, we need several results as follows.
	\begin{Prop}\label{expand}
		Let $F=(F_1,\dots, F_d)$ with $F_i\in\mathbb{D}^{\infty}$, $1\leq i\leq d$. Then, for every $M \geq 2$ and every function $f: \mathbb{R}^d \rightarrow \mathbb{R}$ that is $M$-times continuously differentiable with derivatives having at most polynomial growth, we have
		\begin{align}\label{expand equation}
			\mathrm{E}[\left\langle F, \nabla f(F)\right\rangle _{\mathbb{R}^d}]&=\sum_{s=1}^{M-1}\sum_{m=e_{j_1}+\dots+e_{j_s}, \atop 1\leq j_k\leq d, 1\leq k\leq s} \frac{\kappa_{m}(F)}{(s-1)!} \mathrm{E}\left[\partial^{m}f(F)\right]\\&\quad+\sum_{m=e_{j_1}+\dots+e_{j_M},\atop 1\leq j_k\leq d, 1\leq k\leq M} \mathrm{E}\left[\Gamma_{e_{j_1},\ldots, e_{j_M}}(F) \partial^{m}f(F)\right].
		\end{align}
	\end{Prop}
	\begin{Rem}
		For $d=1$, the expansion of this type in Proposition \ref{expand} can be found in \cite[Proposition 3.11]{Hermine2012}. Proposition \ref{expand} can be seen as an extension of \cite[Proposition 3.11]{Hermine2012} to the multidimensional case. For $d\geq2$, Equation \eqref{expand equation} is new as far as we know.
	\end{Rem}
	\begin{proof}
		Using Lemma \ref{IBP} and Theorem \ref{cumulant and Gamma} repeatedly,
		\begin{align}
			\mathrm{E}\left[F_{j_1}\partial_{j_1}f(F) \right] 
			&=\mathrm{E}\left[F_{j_1}\right] \mathrm{E}\left[\partial_{j_1}f(F) \right]+\mathrm{E}\left[\left\langle D\partial_{j_1}f(F),-DL^{-1}F_{j_1} \right\rangle _{\mathfrak{H}}\right] \\
			&=\kappa_{e_{j_1}}(F) \mathrm{E}\left[\partial_{j_1}f(F) \right]+\sum_{j_2=1}^{d}\mathrm{E}\left[
			\partial^{2}_{j_1,j_2}f(F)\left\langle DF_{j_2},-DL^{-1}F_{j_1} \right\rangle _{\mathfrak{H}} \right] \\
			&=\kappa_{e_{j_1}}(F) \mathrm{E}\left[\partial_{j_1}f(F) \right]+\sum_{j_2=1}^{d}\mathrm{E}\left[
			\partial^{2}_{j_1,j_2}f(F)\right] \mathrm{E}\left[\Gamma_{e_{j_1},e_{j_2}}(F) \right]\\&\quad+\sum_{j_2,j_3=1}^{d}\mathrm{E}\left[
			\partial^{3}_{j_1,j_2,j_3}f(F) \Gamma_{e_{j_1},e_{j_2},e_{j_3}}(F)\right] \\
			&=\kappa_{e_{j_1}}(F) \mathrm{E}\left[\partial_{j_1}f(F) \right]+\sum_{j_2=1}^{d}\kappa_{e_{j_1}+e_{j_2}}(F)\mathrm{E}\left[
			\partial^{2}_{j_1,j_2}f(F)\right] \\&\quad+\sum_{j_2,j_3=1}^{d}\mathrm{E}\left[
			\partial^{3}_{j_1,j_2,j_3}f(F) \Gamma_{e_{j_1},e_{j_2},e_{j_3}}(F)\right] \\
			&=\cdots\\
			&=\kappa_{e_{j_1}}(F) \mathrm{E}\left[\partial_{j_1}f(F) \right]+\sum_{j_2=1}^{d}\kappa_{e_{j_1}+e_{j_2}}(F)\mathrm{E}\left[
			\partial^{2}_{j_1,j_2}f(F)\right]+\cdots\\&\quad+\sum_{j_2,\ldots,j_{M-1}=1}^{d}\frac{\kappa_{e_{j_1}+\cdots+e_{j_{M-1}}}(F)}{(M-2)!}\mathrm{E}\left[
			\partial^{M-1}_{j_1,\ldots,j_{M-1}}f(F)\right]\\&\quad+\sum_{j_2,\ldots,j_M=1}^{d}\mathrm{E}\left[
			\partial^{M}_{j_1,\ldots,j_M}f(F) \Gamma_{e_{j_1},\ldots,e_{j_M}}(F)\right].
		\end{align}
		Therefore,
		\begin{align}
			\mathrm{E}[\left\langle F, \nabla f(F)\right\rangle _{\mathbb{R}^d}]&=\sum_{j_1=1}^{d}\mathrm{E}\left[F_{j_1}\partial_{j_1}f(F) \right]\\
			&=\sum_{j_1=1}^{d}\frac{\kappa_{e_{j_1}}(F)}{(1-1)!} \mathrm{E}\left[\partial_{j_1}f(F) \right]+\sum_{j_1,j_2=1}^{d}\frac{\kappa_{e_{j_1}+e_{j_2}}(F)}{(2-1)!}\mathrm{E}\left[
			\partial^{2}_{j_1,j_2}f(F)\right]+\cdots\\&\quad+\sum_{j_1,\ldots,j_{M-1}=1}^{d}\frac{\kappa_{e_{j_1}+\cdots+e_{j_{M-1}}}(F)}{(M-2)!}\mathrm{E}\left[
			\partial^{M-1}_{j_1,\ldots,j_{M-1}}f(F)\right]\\&\quad+\sum_{j_1,\ldots,j_M=1}^{d}\mathrm{E}\left[
			\partial^{M}_{j_1,\ldots,j_M}f(F) \Gamma_{e_{j_1},\ldots,e_{j_M}}(F)\right]
			\\&=\sum_{s=1}^{M-1}\sum_{m=e_{j_1}+\dots+e_{j_s}, \atop 1\leq j_k\leq d, 1\leq k\leq s} \frac{\kappa_{m}(F)}{(s-1)!} \mathrm{E}\left[\partial^{m}f(F)\right]\\&\quad+\sum_{m=e_{j_1}+\dots+e_{j_M},\atop 1\leq j_k\leq d, 1\leq k\leq M} \mathrm{E}\left[\Gamma_{e_{j_1},\ldots, e_{j_M}}(F) \partial^{m}f(F)\right].
		\end{align}
	\end{proof}
	
	\begin{Prop}\label{estimation of Gamma}
		For each integer $q \geq 2$, there exist positive constants $c_1(q), c_2(q), c_3(q)$ only depending on $q$ such that, for all $F=(I_q(f_1),\ldots,I_q(f_d))$ with $f_i \in \mathfrak{H}^{\odot q}$ and $1\leq i\leq d$, we have
		\begin{align}
			\mathrm{E}\left[\left|\Gamma_{e_i,e_j,e_k}(F)-\frac{1}{2} \kappa_{e_i+e_j+e_k}(F)\right|\right] & \leq c_1(q)\max_{1 \leq i \leq d}\left\lbrace  \kappa_{4e_i}(F)^{\frac{3}{4}}\right\rbrace,  \label{Gamma3}\\
			\mathrm{E}\left[\left|\Gamma_{e_i,e_j,e_k,e_l}(F)\right|\right] & \leq c_2(q) \max_{1 \leq i \leq d}\left\lbrace  \kappa_{4e_i}(F) \right\rbrace,\label{Gamma4} \\
			\mathrm{E}\left[\left|\Gamma_{e_i,e_j,e_k,e_l,e_s}(F)\right|\right] & \leq c_3(q) \max_{1 \leq i \leq d}\left\lbrace  \kappa_{4e_i}(F)^{\frac{5}{4}}\right\rbrace  \label{Gamma5},
		\end{align}
		for any $1\leq i,j,k,l,s \leq d$.
	\end{Prop}
	\begin{Rem}
		See \cite[Proposition 4.3]{Hermine2012} for the estimates of cumulants and related $\Gamma$-random variables for $d=1$.
	\end{Rem}
	\begin{proof}
		By suitable modification to the proof of \cite[Proposition 4.3]{Hermine2012}, we can get the conclusion. Here, we show \eqref{Gamma3} as an instance. According to Equation \eqref{Gamma of integral},
		\begin{equation}\label{gammaijk}
			\Gamma_{e_i,e_j,e_k}(F)=  \sum_{r_2=1}^{q-1}  \sum_{r_3=1}^{\left( 2q-2 r_2\right)  \wedge q} c_{q,l}\left(r_2, r_{3}\right)  I_{3q-2 r_2-2 r_3}\left(\left( f_{i} \tilde{\otimes}_{r_2} f_{j} \right)  \tilde{\otimes}_{r_3} f_{k}\right) ,
		\end{equation}
		where $c_{q,l}\left(r_2, r_{3}\right)$ defined as \eqref{c_q,l} is a constant only depending on $q$, $r_2$ and $r_3$.
		By Theorem \ref{cumulant and Gamma}, we have 
		\begin{equation}
			\mathrm{E}\left[ \Gamma_{e_i,e_j,e_k}(F)\right] =\frac{1}{2}\kappa_{e_i+e_j+e_k}(F).
		\end{equation}
		Therefore, the random variable $\Gamma_{e_i,e_j,e_k}(F)-\frac{1}{2} \kappa_{e_i+e_j+e_k}(F)$ is obtained by restricting the sum in \eqref{gammaijk} to the terms such that $2r_2+2r_3<3 q$. Combining the fact that there exists a constant $c(q)$ only depending on $q$ such that
		\begin{equation}
			\max_{1 \leq r \leq q-1}\left\|f_i\otimes_{r} f_i\right\|_{\mathfrak{H}^{\otimes(2q-2r)}}^2\leq c(q)\kappa_{4e_i}(F),
		\end{equation}
		which is from \cite[Equation (5.2.6)]{nourdin2012normal}, it suffices to show that for $r_2$ and $r_3$ satisfying $1\leq r_2\leq q-1$, $1\leq r_3\leq \left( 2q-2 r_2\right)  \wedge q$ and $2r_2+2r_3<3 q$,
		\begin{equation}
			\left\| \left( f_{i} \tilde{\otimes}_{r_2} f_{j} \right)  \tilde{\otimes}_{r_3} f_{k}\right\| _{\mathfrak{H}^{\otimes(3q-2r_2-2r_3)}}\leq \max_{1 \leq i \leq d}	\max_{1 \leq r \leq q-1}\left\|f_i\otimes_{r} f_i\right\|_{\mathfrak{H}^{\otimes(2q-2r)}}^{\frac{3}{2}}.
		\end{equation}
		Firstly we assume that $r_3<q$, then both $q-r_2$ and $q-r_3$ belong to $\left\lbrace 1,\ldots ,q-1\right\rbrace $. By Cauchy-Schwarz inequality (or see \cite[Equation (4.3), Equation (4.4)]{Hermine2012}), we get that
		\begin{align}
			\left\| \left( f_{i} \tilde{\otimes}_{r_2} f_{j} \right)  \tilde{\otimes}_{r_3} f_{k}\right\| _{\mathfrak{H}^{\otimes(3q-2r_2-2r_3)}}&\leq \left\| \left( f_{i} \tilde{\otimes}_{r_2} f_{j} \right)  \otimes_{r_3} f_{k}\right\| _{\mathfrak{H}^{\otimes(3q-2r_2-2r_3)}}\\
			&\leq \left\|  f_{i} \tilde{\otimes}_{r_2} f_{j}\right\| _{\mathfrak{H}^{\otimes(2q-2r_2)}}\sqrt{\left\| f_k  \otimes_{q-r_3} f_{k}\right\| _{\mathfrak{H}^{\otimes2r_3}}}\\
			&\leq \sqrt{\left\| f_i  \otimes_{q-r_2} f_{i}\right\| _{\mathfrak{H}^{\otimes2r_2}}}\sqrt{\left\| f_j \otimes_{q-r_2} f_{j}\right\| _{\mathfrak{H}^{\otimes2r_2}}}\sqrt{\left\| f_k  \otimes_{q-r_3} f_{k}\right\| _{\mathfrak{H}^{\otimes2r_3}}}\\
			&\leq\max_{1 \leq i \leq d}	\max_{1 \leq r \leq q-1}\left\|f_i\otimes_{r} f_i\right\|_{\mathfrak{H}^{\otimes(2q-2r)}}^{\frac{3}{2}}.
		\end{align}
		Now we consider the case when $r_3=q$ and $r_2< \frac{q}{2}$. Then 
		\begin{equation}
			\left\| \left( f_{i} \tilde{\otimes}_{r_2} f_{j} \right)  \tilde{\otimes}_{r_3} f_{k}\right\| _{\mathfrak{H}^{\otimes(3q-2r_2-2r_3)}}=\left\langle f_{i} \tilde{\otimes}_{r_2} f_{j} ,f_k \right\rangle _{\mathfrak{H}^{\otimes q}}	
		\end{equation}
		defines a function of $q-2r_2$ variables. By the similar argument of the proof of \cite[Equation  (4.6)]{Hermine2012}, we know that
		\begin{equation}
			\left\| \left( f_{i} \tilde{\otimes}_{r_2} f_{j} \right)  \tilde{\otimes}_{r_3} f_{k}\right\| _{\mathfrak{H}^{\otimes(3q-2r_2-2r_3)}}\leq \max_{1 \leq r \leq q-1}\sqrt{\left\|f_i\otimes_{r} f_i\right\|_{\mathfrak{H}^{\otimes(2q-2r)}}}\left\|f_j\otimes_{r} f_k\right\|_{\mathfrak{H}^{\otimes(2q-2r)}}.
		\end{equation}
		Using \cite[Equation (4.4)]{Hermine2012} again, we obtain that 
		\begin{align}
			&\left\| \left( f_{i} \tilde{\otimes}_{r_2} f_{j} \right)  \tilde{\otimes}_{r_3} f_{k}\right\| _{\mathfrak{H}^{\otimes(3q-2r_2-2r_3)}}\\\leq& \,\max_{1 \leq r \leq q-1}\sqrt{\left\|f_i\otimes_{r} f_i\right\|_{\mathfrak{H}^{\otimes(2q-2r)}}}\sqrt{\left\|f_j\otimes_{r} f_j\right\|_{\mathfrak{H}^{\otimes(2q-2r)}}}\sqrt{\left\|f_k\otimes_{r} f_k\right\|_{\mathfrak{H}^{\otimes(2q-2r)}}}\\\leq&\, \max_{1 \leq i \leq d}	\max_{1 \leq r \leq q-1}\left\|f_i\otimes_{r} f_i\right\|_{\mathfrak{H}^{\otimes(2q-2r)}}^{\frac{3}{2}}.
		\end{align}
		Then we finish the proof.
	\end{proof} 
	
	Inspired by \cite{Hermine2012,NP2015}, we next construct several specific test functions that will be utilized in the proof of the lower bound in Theorem \ref{Main result1}. Let $$a=\exp{\left\lbrace -\frac{1}{2}\max\limits_{t\in\left\lbrace-1,0,1\right\rbrace  ^d}t^TCt\right\rbrace}.$$ Define 
	\begin{align}
		g_t(x)&=a\exp{\left\lbrace \frac{1}{2}t^TCt\right\rbrace}\sin \left(\left\langle t,x\right\rangle _{\mathbb{R}^d}\right),\\
		h_t(x)&=a\exp{\left\lbrace \frac{1}{2}t^TCt\right\rbrace}\cos \left(\left\langle t,x\right\rangle _{\mathbb{R}^d}\right).
	\end{align}
	\begin{Lemma}\label{specific test function}
		Fix $1\leq i,j,k\leq d$ satisfying $i\neq j,k$ and $j\neq k$, define $h_i(x), g_i(x),g_{ij}(x), g_{ijk}(x): \mathbb{R}^{d}\rightarrow \mathbb{R}$ as
		\begin{align}
			h_i(x)&=h_{e_i}(x)=ae^{\frac{1}{2}C_{ii}}\cos x_i,\\
			g_i(x)&=g_{e_i}(x)=ae^{\frac{1}{2}C_{ii}}\sin x_i,\\
			g_{ij}(x)&=\frac{1}{4}\left( g_{e_i-e_j}(x)-g_{e_i+e_j}(x)+2g_{e_j}(x)\right), \\
			g_{ijk}(x)&=\frac{1}{12}\left( g_{e_i+e_j-e_k}(x)-g_{e_i+e_j+e_k}(x)-4g_{ik}(x)-4g_{jk}(x)+2g_{e_k}(x)\right). 
		\end{align}
		Then $h_i(x), g_i(x),g_{ij}(x), g_{ijk}(x)$ are bounded by one and infinitely continuously differentiable with all derivatives bounded by one, and satisfy
		\begin{align}
			\mathrm{E}\left[\partial^m U_{h_i, C}(Z)\right]&=\left\lbrace 
			\begin{aligned}
				&\frac{a}{|m|}(-1)^\frac{|m|}{2},& m= |m|e_i, |m|=0,2,4,\cdots,\\
				&0, &otherwise,
			\end{aligned}\right.\label{equation1}\\
			\mathrm{E}\left[\partial^m U_{g_i, C}(Z)\right]&=\left\lbrace 
			\begin{aligned}
				&\frac{a}{|m|}(-1)^\frac{|m|-1}{2},& m= |m|e_i, |m|=1,3,5,\cdots,\\
				&0, &otherwise,
			\end{aligned}\right.\label{equation2}
		\end{align}
	    \begin{align}
			\mathrm{E}\left[\partial^m U_{g_{ij}, C}(Z)\right]&=\left\lbrace 
			\begin{aligned}
				&\frac{a}{2|m|}(-1)^{\frac{|m|+1}{2}},& m=m_ie_i+m_je_j, m_i>0, m_j \mbox{ and } |m| \mbox{ are odd,} 			\\
				&0, &otherwise,
			\end{aligned}\right.\label{equation3}
		\end{align}
		and
		\begin{align}\label{equation4}
			&\mathrm{E}\left[\partial^m U_{g_{ijk}, C}(Z)\right]\\=&\,\left\lbrace 
			\begin{aligned}
				&\frac{a}{6|m|}(-1)^{\frac{|m|+1}{2}},& m=m_ie_i+m_je_j+m_ke_k, m_i,m_j>0, m_k \mbox{ and } |m| \mbox{ are odd,}  \\
				&0, &otherwise,
			\end{aligned}\right.
		\end{align}
		where $Z\sim \mathcal{N}_d(0,C)$ and $U_{g, C}(x)$ is defined as \eqref{solution of Stein}.
	\end{Lemma}
	
	\begin{proof}
		Firstly, it is obvious that $h_i(x), g_i(x), g_{ij}(x), g_{ijk}(x)$ are bounded by one and infinitely continuously differentiable with all derivatives bounded by one.
		
		For $Z=(Z_1,\ldots, Z_d)\sim \mathcal{N}_d(0,C)$, for any $t\in \mathbb{R}^d$, we have
		\begin{equation}
			e^{-\frac{1}{2}t^TCt}= \mathrm{E}\left[ e^{\mathrm{i}\left\langle t,Z\right\rangle _{\mathbb{R}^d}}\right] =\mathrm{E}\left[\cos \left( \left\langle t,Z\right\rangle _{\mathbb{R}^d}\right) \right]+\mathrm{i}\mathrm{E}\left[\sin \left(\left\langle t,Z\right\rangle _{\mathbb{R}^d}\right) \right],
		\end{equation}
		that is, for any $t\in \mathbb{R}^d$,
		\begin{equation}
			\mathrm{E}\left[\sin \left(\left\langle t,Z\right\rangle _{\mathbb{R}^d}\right) \right]=0, \quad \mathrm{E}\left[\cos \left(\left\langle t,Z\right\rangle _{\mathbb{R}^d}\right) \right]=e^{-\frac{1}{2}t^TCt}.
		\end{equation}
		Fix $1\leq i\leq d$, let $h_i(x): \mathbb{R}^{d}\rightarrow \mathbb{R}, h_i(x)=h_{e_i}(x)=ae^{\frac{1}{2}C_{ii}}\cos x_i$,
		\begin{equation}
			\partial^{m} h_i(x)=\left\lbrace
			\begin{aligned}
				& ae^{\frac{1}{2}C_{ii}}(-1)^{\frac{|m|+1}{2}}\sin x_i , & m=|m|e_i, |m|=1,3,5,\cdots,\\
				& ae^{\frac{1}{2}C_{ii}}(-1)^{\frac{|m|}{2}}\cos x_i ,& m=|m|e_i, |m|=0,2,4,\cdots,\\
				& 0 , &otherwise .		    
			\end{aligned} \right. 
		\end{equation}
		Then, by \eqref{Expectation of derivative of U},
		\begin{equation}
			\mathrm{E}\left[\partial^m U_{h_i, C}(Z)\right]=\frac{1}{|m|} \mathrm{E}\left[\partial^m h_i(Z)\right]=\left\lbrace 
			\begin{aligned}
				&\frac{a}{|m|}(-1)^\frac{|m|}{2},& m= |m|e_i, |m|=0,2,4,\cdots,\\
				&0, &otherwise.
			\end{aligned}\right. 
		\end{equation}
	By a similar argument, we get \eqref{equation2} and for fixed $1\leq i,j\leq d$ satisfying $i\neq j$,
		\begin{align}
			\mathrm{E}\left[\partial^m U_{g_{e_i-e_j}, C}(Z)\right]&=\frac{1}{|m|} \mathrm{E}\left[\partial^m g_{e_i-e_j}(Z)\right]\\&=\left\lbrace 
			\begin{aligned}
				&\frac{a}{|m|}(-1)^{\frac{|m|-1}{2}+m_j},& m=m_ie_i+m_je_j, |m|\mbox{ is odd},				\\
				&0, &otherwise,
			\end{aligned}\right. \\
			\mathrm{E}\left[\partial^m U_{g_{e_i+e_j}, C}(Z)\right]&=\frac{1}{|m|} \mathrm{E}\left[\partial^m g_{e_i+e_j}(Z)\right]\\&=\left\lbrace 
			\begin{aligned}
				&\frac{a}{|m|}(-1)^{\frac{|m|-1}{2}},& m=m_ie_i+m_je_j, |m|\mbox{ is odd},				\\
				&0, &otherwise.
			\end{aligned}\right. 
		\end{align}
		Then for $g_{ij}(x)=\frac{1}{4}\left( g_{e_i-e_j}(x)-g_{e_i+e_j}(x)+2g_{e_j}(x)\right)$, by \eqref{Derivative of U},
		\begin{align}
			&\mathrm{E}\left[\partial^m U_{g_{ij}, C}(Z)\right]=\frac{1}{4}\left( \mathrm{E}\left[\partial^m U_{g_{e_i-e_j}, C}(Z)\right]- \mathrm{E}\left[\partial^m U_{g_{e_i+e_j}, C}(Z)\right]+2\mathrm{E}\left[\partial^m U_{g_j, C}(Z)\right] \right) \\=&\,\left\lbrace 
			\begin{aligned}
				&\frac{a}{2|m|}(-1)^{\frac{|m|+1}{2}},& m=m_ie_i+m_je_j, m_i>0, m_j \mbox{ is odd, } |m|\mbox{ is odd},				\\
				&0, &otherwise.
			\end{aligned}\right. 
		\end{align}
		Similarly, we can obtain \eqref{equation4}.

	\end{proof}

	We are now turning to the proof of Theorem \ref{Main result1}.
	\begin{proof}
	\emph{Upper bound.}	By Stein's equation \eqref{Stein's equation} and Proposition \ref{expand},
		\begin{align}\label{expansion for F_n}
			\mathrm{E}\left[ g(Z)\right] - \mathrm{E}\left[ g(F_n)\right]
			&= \mathrm{E}\left[ \left\langle F_n, \nabla U_{g, C}(F_n)\right\rangle_{\mathbb{R}^d} \right] -\mathrm{E} \left[ \left\langle C, \operatorname{Hess} U_{g, C}(F_n)\right\rangle_{\mathrm{HS}}\right] \\
			&=\sum_{s=3}^{M-1}\sum_{m=e_{j_1}+\dots+e_{j_s}, \atop 1\leq j_k\leq d, 1\leq k\leq s} \frac{\kappa_{m}(F_n)}{(s-1)!} \mathrm{E}\left[\partial^{m}U_{g, C}(F_n)\right]\\&\quad+\sum_{m=e_{j_1}+\dots+e_{j_M},\atop 1\leq j_k\leq d, 1\leq k\leq M} \mathrm{E}\left[\Gamma_{e_{j_1},\ldots, e_{j_M}}(F_n) \partial^{m}U_{g, C}(F_n)\right].
		\end{align}
		Take $M=4$, 
		\begin{align}
			\mathrm{E}\left[ g(Z)\right] - \mathrm{E}\left[ g(F_n)\right]&=\frac{1}{2}\sum_{m=e_{i}+e_{j}+e_k,  \atop 1\leq i,j,k\leq d} \kappa_{m}(F_n) \mathrm{E}\left[\partial^{m}U_{g, C}(F_n)\right]\\&\quad+
			\sum_{m=e_{i}+e_{j}+e_k+e_l,
				\atop 1\leq i,j,k,l\leq d} \mathrm{E}\left[\Gamma_{e_{i},e_{j},e_k,e_l}(F_n) \partial^{m}U_{g, C}(F_n)\right].
		\end{align}
		Combining \eqref{Bound of derivative of U} and Proposition \ref{estimation of Gamma},
		\begin{align}
			\left| \mathrm{E}\left[ g(Z)\right] - \mathrm{E}\left[ g(F_n)\right]\right| &\leq \frac{1}{2}\sup_{x\in{\mathbb{R}}^d,|m|=3}\left| \partial^{m}U_{g, C}(x)\right| \sum_{m=e_{i}+e_{j}+e_k} \left| \kappa_{m}(F_n)\right|  \\&\quad+\sup_{x\in{\mathbb{R}}^d,|m|=4}\left| \partial^{m}U_{g, C}(x)\right|
			\sum_{m=e_{i}+e_{j}+e_k+e_l} \mathrm{E}\left[\left| \Gamma_{e_{i},e_{j},e_k,e_l}(F_n)\right|  \right]
			\\&\leq \frac{d^3}{6} \sum_{|m|=3}\left|  \kappa_{m}(F_n)\right| +\frac{d^4}{4}c_2(q)\sum_{i=1}^{d}\kappa_{4e_i}(F_n)
			\\&\leq \max\left\lbrace \frac{d^3}{3}, \frac{d^4c_2(q)}{2}\right\rbrace  \max\left\lbrace \sum_{|m|=3}\left|  \kappa_{m}(F_n)\right|, \sum_{i=1}^{d}\kappa_{4e_i}(F_n)\right\rbrace.
		\end{align}

		\emph{Lower bound.} Take $M=5$ in \eqref{expansion for F_n}, we have 
		\begin{align}
			&\mathrm{E}\left[ g(Z)\right] - \mathrm{E}\left[ g(F_n)\right]
			\\=&\,\frac{1}{2}\sum_{m=e_{i}+e_{j}+e_k,\atop 1\leq i,j,k\leq d} \kappa_{m}(F_n) \mathrm{E}\left[\partial^{m}U_{g, C}(F_n)\right]+\frac{1}{6}\sum_{m=e_{i}+e_{j}+e_k+e_l, \atop 1\leq i,j,k,l\leq d} \kappa_{m}(F_n) \mathrm{E}\left[\partial^{m}U_{g, C}(F_n)\right]\\&\,+
			\sum_{m=e_{j_1}+\dots+e_{j_5},\atop 1\leq j_k\leq d, 1\leq k\leq 5} E\left[\Gamma_{e_{j_1},\ldots, e_{j_5}}(F_n) \partial^{m}U_{g, C}(F_n)\right].
		\end{align}
		Replace the test function $g$ with $h_i$, then we get that
		\begin{align}
			&\left| \mathrm{E}\left[ h_i(Z)\right] - \mathrm{E}\left[ h_i(F_n)\right]-\frac{a}{24} \kappa_{4e_i}(F_n) \right|\\
			=&\,\left| \mathrm{E}\left[ h_i(Z)\right] - \mathrm{E}\left[ h_i(F_n)\right]-\sum_{m=e_{i}+e_{j}+e_k+e_l} \frac{\kappa_{m}(F_n)}{6} \mathrm{E}\left[\partial^{m}U_{h_i, C}(Z)\right]\right| \\
			=&\,\left| \sum_{m=e_{i}+e_{j}+e_k,\atop 1\leq i,j,k\leq d} \frac{\kappa_{m}(F_n)}{2} \mathrm{E}\left[\partial^{m}U_{h_i,C}(F_n)\right]\right. \\&\left. \,+\sum_{m=e_{i}+e_{j}+e_k+e_l, \atop 1\leq i,j,k,l\leq d} \frac{\kappa_{m}(F_n)}{6} \left( \mathrm{E}\left[\partial^{m}U_{h_i,C}(F_n)\right]-\mathrm{E}\left[\partial^{m}U_{h_i,C}(Z)\right]\right)\right. \\&\,+
			\left. \sum_{m=e_{j_1}+\dots+e_{j_5},\atop 1\leq j_k\leq d, 1\leq k\leq 5} \mathrm{E}\left[\Gamma_{e_{j_1},\ldots, e_{j_5}}(F_n) \partial^{m}U_{h_i, C}(F_n)\right]\right| \\
			\leq &\, \frac{1}{2}\sum_{m=e_{i}+e_{j}+e_k} \left| \kappa_{m}(F_n)\right|  \left|\mathrm{E}\left[ \partial^{m}U_{h_i, C}(F_n)\right]\right|\\
			&\, + \frac{1}{6}\sum_{m=e_{i}+e_{j}+e_k+e_l} \left| \kappa_{m}(F_n)\right| \left|  \mathrm{E}\left[\partial^{m}U_{h_i, C}(F_n)\right]-\mathrm{E}\left[ \partial^{m}U_{h_i, C}(Z) \right]\right|\\
			&\,+ \sum_{m=e_{j_1}+\dots+e_{j_5},\atop 1\leq j_k\leq d, 1\leq k\leq 5} \left\|  \partial^{m}U_{h_i, C}\right\| _{\infty} \mathrm{E}\left[\left| \Gamma_{e_{j_1},\ldots, e_{j_5}}(F_n)\right| \right] \\
			\leq&\, \max\left\lbrace \sum_{|m|=3}\left|  \kappa_{m}(F_n)\right|, \sum_{i=1}^{d}\kappa_{4e_i}(F_n)\right\rbrace\left[ \frac{1}{2}\sum_{m=e_{i}+e_{j}+e_k}  \left|\mathrm{E}\left[ \partial^{m}U_{h_i, C}(F_n)\right]\right|\right. \\&\left. \,+c_2(q)\sum_{m=e_{i}+e_{j}+e_k+e_l} \left|  \mathrm{E}\left[\partial^{m}U_{h_i, C}(F_n)\right]-\mathrm{E}\left[ \partial^{m}U_{h_i, C}(Z) \right]\right|\right. \\&\left. \,+\frac{c_3(q)d^5}{5}\left(\sum_{i=1}^{d}\kappa_{4e_i}(F_n) \right) ^{\frac{1}{4}}\right] .
		\end{align}
		As $n\rightarrow\infty$, we have 
		$\mathrm{E}\left[\partial^{m}U_{h_i, C}(F_n)\right]\rightarrow \mathrm{E}\left[\partial^{m}U_{h_i, C}(Z)\right]=0$ for $|m|=3$, $\mathrm{E}\left[\partial^{m}U_{h_i, C}(F_n)\right]-\mathrm{E}\left[ \partial^{m}U_{h_i, C}(Z) \right]\rightarrow 0$, and $\sum_{i=1}^{d}\kappa_{4e_i}(F_n)\rightarrow 0$. Therefore, set 
		\begin{equation}
			\bar{d}=2d+d(d-1)+\frac{d(d-1)(d-2)}{6}	,\quad c_1=\frac{a}{36\left( \bar{d}+1\right) },
		\end{equation}
		we have that for $n$ large enough,
		\begin{equation}
			\left| \mathrm{E}\left[ h_i(Z)\right] - \mathrm{E}\left[ h_i(F_n)\right]-\frac{a}{24} \kappa_{4e_i}(F_n) \right|\leq \frac{c_1}{\bar{d}}M(F_n),
		\end{equation}
		which implies that 
		\begin{equation}
			\left| \mathrm{E}\left[ h_i(Z)\right] - \mathrm{E}\left[ h_i(F_n)\right] \right|\geq \frac{a}{24} \kappa_{4e_i}(F_n)-\frac{c_1}{\bar{d}}M(F_n), \quad 1 \leq i\leq d.
		\end{equation}
		Similarly, for $1 \leq i,j,k\leq d$,
		\begin{align}
			\left| \mathrm{E}\left[ g_i(Z)\right] - \mathrm{E}\left[ g_i(F_n)\right] \right|&\geq \frac{a}{6} \left| \kappa_{3e_i}(F_n)\right| -\frac{c_1}{\bar{d}}M(F_n),\\
			\left| \mathrm{E}\left[ g_{ij}(Z)\right] - \mathrm{E}\left[ g_{ij}(F_n)\right] \right|&\geq \frac{a}{12} \left| \kappa_{2e_i+e_j}(F_n)\right| -\frac{c_1}{\bar{d}}M(F_n), \, i\neq j, \\
			\left| \mathrm{E}\left[ g_{ijk}(Z)\right] - \mathrm{E}\left[ g_{ijk}(F_n)\right] \right|&\geq \frac{a}{36} \left| \kappa_{e_i+e_j+e_k}(F_n)\right| -\frac{c_1}{\bar{d}}M(F_n), \, i\neq j, k \mbox{ and } j\neq k.
		\end{align}
		Then 
		\begin{align}
			\bar{d}\rho(F_n,Z)&\geq \sum_{i=1}^{d}\left| \mathrm{E}\left[ h_i(Z)\right] - \mathrm{E}\left[ h_i(F_n)\right]\right|+\sum_{i=1}^{d}\left| \mathrm{E}\left[ g_i(Z)\right] - \mathrm{E}\left[g_i(F_n)\right]\right|\\& + \sum_{i=1}^{d}\sum_{j\neq i}\left| \mathrm{E}\left[ g_{ij}(Z)\right] - \mathrm{E}\left[g_{ij}(F_n)\right]\right|+\sum_{i=1}^{d}\sum_{j= i+1}^{d}\sum_{k=j+1}^{d}\left| \mathrm{E}\left[ g_{ijk}(Z)\right] - \mathrm{E}\left[g_{ijk}(F_n)\right]\right|\\
			&\geq \left( \bar{d}+1\right) c_1\left(\sum_{|m|=3}\left|  \kappa_{m}(F_n)\right|+ \sum_{i=1}^{d}\kappa_{4e_i}(F_n) \right) -c_1 M(F_n)\\
			&\geq \left( \bar{d}+1\right) c_1M(F_n) -c_1 M(F_n)\\
			&=\bar{d}c_1M(F_n).
		\end{align}
		That is,
		\begin{align}
			\rho(F_n,Z)&\geq c_1 M(F_n).
		\end{align}
	\end{proof}

	\section{Applications}\label{Section 4}
	\subsection{Application for complex Wiener-It\^o integral}\label{Section 4.1}
	
	We define the distribution of a complex random variable $F=F_1+\mathrm{i}F_2$ as the distribution of two-dimensional random vector $(F_1,F_2)$. Then the distance between the distributions of two complex random variables $F=F_1+\mathrm{i}F_2$ and $G=G_1+\mathrm{i}G_2$ is actually the distance between the distributions of two two-dimensional random vectors $(F_1,F_2)$ and $(G_1,G_2)$, namely, we take $d=2$ in \eqref{distance} and define
	\begin{equation}
		\rho(F,G)=\sup\left\lbrace\left|\mathrm{E}\left[ g(F_1,F_2) \right] - \mathrm{E}\left[ g(G_1,G_2) \right]   \right|  \right\rbrace,	
	\end{equation} 
	where $g:\mathbb{R}^2\rightarrow \mathbb{R}$ runs over the class of all four-times continuously differentiable functions such that $g$ and all of its derivatives of order up to four are bounded by one. Define the covariance matrix of the complex random variable $F=F_1+\mathrm{i}F_2$ as the covariance matrix of the two-dimensional random vector $(F_1,F_2)$. We write $AF$ to denote $A\begin{pmatrix}
		F_1\\F_2
	\end{pmatrix}$ for any $2\times2$ matrix $A$. 
	
	For a sequence of complex random variables $\left\lbrace F_n=F_{n,1}+\mathrm{i}F_{n,2}: n\geq1\right\rbrace$, let 
	\begin{equation}
		M^{'}(F_n)=\max\left\lbrace \left| \mathrm{E}\left[F_n^3 \right] \right|, \left|\mathrm{E}\left[F_n^2\bar{F_n} \right] \right|, \mathrm{E}\left[\left| F_n \right| ^4 \right]-2\left(\mathrm{E}\left[ \left| F_n \right| ^2\right]  \right) ^2-\left| \mathrm{E}\left[ F_n^2\right] \right| ^2 \right\rbrace.
	\end{equation}
	\begin{Thm}\label{Main result2}
		Consider a sequence of complex Wiener-It\^o integrals $\left\lbrace F_n= \mathcal{I}_{p,q}(f_n): n\geq1\right\rbrace$, where $f_n\in \mathfrak{H}_{\mathbb{C}}^{\odot p}\otimes\mathfrak{H}_{\mathbb{C}}^{\odot q}$ and $p+q\geq 2$. Suppose that $F_n$ converges in distribution to a complex normal random variable $Z$ with the same covariance matrix as $F_n$. Then there exist two finite constants $0<c_1<c_2$ only depending on $p,q$ such that for $n$ large enough,
		\begin{equation}
			c_1 M^{'}(F_n) \leq \rho\left(F_n, Z\right) \leq c_2 M^{'}(F_n).
		\end{equation}
	\end{Thm}
	
	\begin{proof}
		Assume $F_n=F_{n,1}+\mathrm{i}F_{n,2}$. According to \cite[Theorem 3.3]{chen2017fourth}, $\left\lbrace (F_{n,1},F_{n,2}): n\geq1 \right\rbrace $ actually is a sequence of two-dimensional random vectors of which components live in the $(p+q)$-th Wiener chaos of the real Gaussian isonormal process over $\mathfrak{H}\oplus\mathfrak{H}$. Combining Theorem \ref{Main result1} and the fact that 
		\begin{equation}
			M\left((F_{n,1},F_{n,2}) \right) \asymp M^{'}(F_n),
		\end{equation}
	which is from the following Lemma \ref{4real to complex} and Lemma \ref{3real to complex}, we get the conclusion.
	\end{proof}

Using the similar argument as in the proof of Theorem \ref{Main result2}, we can extend Theorem \ref{Main result2} to the case that the covariance matrix of $F_n$, denoted by $C_n$, converges to $C$ in the meaning of $\|C_n-C\|_{\mathrm{HS}}\rightarrow 0$ as $n\rightarrow\infty$.
\begin{Prop}
	 Let $\left\lbrace F_n= \mathcal{I}_{p,q}(f_n): n\geq1\right\rbrace$ be a sequence of complex Wiener-It\^o integrals, where $f_n\in \mathfrak{H}_{\mathbb{C}}^{\odot p}\otimes\mathfrak{H}_{\mathbb{C}}^{\odot q}$ and $p+q\geq 2$. Suppose that $C_n$, the covariance matrix of $F_n$, converges to $C$ in the meaning of $\|C_n-C\|_{\mathrm{HS}}\rightarrow 0$ as $n\rightarrow\infty$. 
	\begin{enumerate}[(i)]
		\item If $C$ is invertible, we set $F_n^{'}=C^{\frac{1}{2}}C_n^{-\frac{1}{2}}F_n$ and assume that $F_n^{'}$ converges in distribution to a complex normal random variable $Z$ with covariance matrix $C$. Then for $n$ large enough,
		\begin{equation}
			\rho(F_n^{'}, Z)\asymp M^{'}(F_n^{'}).
		\end{equation}
		\item If $C$ is not invertible, suppose that $\left\lbrace F_n:n\geq1 \right\rbrace $ is asymptotically close to normal. That is, $\rho(F_n,Z_n)\rightarrow0$, where $Z_n$ is a complex normal random variable of which covariance matrix coincides with that of $F_n$. Then for $n$ large enough,
		\begin{equation}
			\rho(F_n,Z_n)\asymp M^{'}(F_n).
		\end{equation}
	\end{enumerate}
\end{Prop}

In the following Lemma \ref{4real to complex} and Lemma \ref{3real to complex}, we prove that $M\left((F_{n,1},F_{n,2}) \right) \asymp M^{'}(F_n)$ for $F_n=\I_{p,q}(f_n)=F_{n,1}+\mathrm{i}F_{n,2}$ with $f_n\in \mathfrak{H}_{\mathbb{C}}^{\odot p}\otimes\mathfrak{H}_{\mathbb{C}}^{\odot q}$ and $p+q\geq 2$.

	\begin{Lemma}\label{4real to complex}
	For a complex Wiener-It\^o integral $ F= \mathcal{I}_{p,q}(f)=F_{1}+\mathrm{i}F_{2}$ with $f\in \mathfrak{H}_{\mathbb{C}}^{\odot p}\otimes\mathfrak{H}_{\mathbb{C}}^{\odot q}$ and $p+q\geq 2$, denote $\tilde{F}$ the two-dimensional random vector $\left( F_{1},F_{2}\right) $. Then
	\begin{equation*}
		\sum_{i=1}^{2}\kappa_{4e_i}(\tilde{F})\leq \mathrm{E}\left[\left| F \right| ^4 \right]-2\left(\mathrm{E}\left[ \left| F \right| ^2\right]  \right) ^2-\left| \mathrm{E}\left[ F^2\right] \right| ^2\leq c\sum_{i=1}^{2}\kappa_{4e_i}(\tilde{F}),
	\end{equation*}
	where $\sum_{i=1}^{2}\kappa_{4e_i}(\tilde{F})= \mathrm{E}\left[F_{1}^4 \right] -3\left( \mathrm{E}\left[ F_{1}^2\right] \right) ^2+\mathrm{E}\left[F_{2}^4 \right] -3\left( \mathrm{E}\left[ F_{2}^2\right] \right) ^2$ and $c>1$ is a positive constant only depending on $p+q$.
\end{Lemma}

\begin{proof}
	Calculating directly, we get that
	\begin{align}
		&\mathrm{E}\left[\left| F \right| ^4 \right]-2\left(\mathrm{E}\left[ \left| F \right| ^2\right]  \right) ^2-\left| \mathrm{E}\left[ F^2\right] \right| ^2\\
		=&\,\mathrm{E}\left[F_{1}^{4}\right]-3\left(\mathrm{E}\left[F_{1}^{2}\right]\right)^{2}+\mathrm{E}\left[F_{2}^{4}\right]-3\left(\mathrm{E}\left[F_{2}^{2}\right]\right)^{2}\\
		&\,+2\left(\mathrm{E}\left[F_{1}^{2}F_{2}^{2} \right]-\mathrm{E}\left[ F_{1}^{2}\right] \mathrm{E}\left[ F_{2}^{2}\right]-2\left(\mathrm{E}\left[ F_{1}F_{2}\right]  \right) ^2  \right) .
	\end{align}
	According to \cite[Theorem 3.3]{chen2017fourth}, $F_1$ and $F_2$ are two $(p+q)$-th Wiener It\^o integrals with respect to the real Gaussian isonormal process over $\mathfrak{H}\oplus\mathfrak{H}$. Using product formula \eqref{Product_formula}, isometry properties and some combinatorics (see \cite[Lemma 4.8]{chen2017fourth}), we know that
	\begin{equation}
		2\left(\mathrm{E}\left[F_{1}^{2}F_{2}^{2} \right]-\mathrm{E}\left[ F_{1}^{2}\right] \mathrm{E}\left[ F_{2}^{2}\right]-2\left(\mathrm{E}\left[ F_{1}F_{2}\right]  \right) ^2\right) \geq0.
	\end{equation}
	Thus,
	\begin{equation}
		\mathrm{E}\left[\left| F \right| ^4 \right]-2\left(\mathrm{E}\left[ \left| F \right| ^2\right]  \right) ^2-\left| \mathrm{E}\left[ F^2\right] \right| ^2\geq \mathrm{E}\left[F_{1}^{4}\right]-3\left(\mathrm{E}\left[F_{1}^{2}\right]\right)^{2}+\mathrm{E}\left[F_{2}^{4}\right]-3\left(\mathrm{E}\left[F_{2}^{2}\right]\right)^{2}.
	\end{equation}
	On the other hand, by Theorem \ref{cumulant and Gamma} and Equation \eqref{Gamma4},
	\begin{align}
		&\mathrm{E}\left[F_{1}^{2}F_{2}^{2} \right]-\mathrm{E}\left[ F_{1}^{2}\right] \mathrm{E}\left[ F_{2}^{2}\right]-2\left(\mathrm{E}\left[ F_{1}F_{2}\right]  \right) ^2\\ = &\, \kappa_{2e_1+2e_2}(\tilde{F})=6\mathrm{E}\left[ \Gamma_{e_{1},e_{1},e_2,e_2}(\tilde{F})\right] 
		\\\leq &\, 6\mathrm{E}\left[\left|  \Gamma_{e_{1},e_{1},e_2,e_2}(\tilde{F})\right| \right]	
		\leq 6c_2(p+q)\sum_{i=1}^{2}\kappa_{4e_i}(\tilde{F}) ,
	\end{align}
	which means that 
	\begin{align}
		&\mathrm{E}\left[\left| F \right| ^4 \right]-2\left(\mathrm{E}\left[ \left| F \right| ^2\right]  \right) ^2-\left| \mathrm{E}\left[ F^2\right] \right| ^2\\\leq&\, (1+12c_2(p+q))\left(  \mathrm{E}\left[F_{1}^{4}\right]-3\left(\mathrm{E}\left[F_{1}^{2}\right]\right)^{2}+\mathrm{E}\left[F_{2}^{4}\right]-3\left(\mathrm{E}\left[F_{2}^{2}\right]\right)^{2}\right).
	\end{align}
	Then we complete the proof.
\end{proof}

\begin{Lemma}\label{3real to complex}
	For a complex random variable $ F=F_{1}+\mathrm{i}F_{2} $, it holds that
\begin{align}
	\frac{1}{8}\left( \left| \mathrm{E}\left[F^3 \right] \right| +\left|\mathrm{E}\left[F^2\bar{F} \right] \right|\right)  &\leq \left| \mathrm{E}\left[F_{1}^3 \right] \right|+	\left| \mathrm{E}\left[F_{2}^3 \right] \right| +\left| \mathrm{E}\left[F_{1}^2F_{2} \right] \right|+ \left| \mathrm{E}\left[F_{1}F_{2}^2 \right] \right|\\&\leq \sqrt{2}\left( \left| \mathrm{E}\left[F^3 \right] \right| +\left|\mathrm{E}\left[F^2\bar{F} \right] \right|\right) .
\end{align}
\end{Lemma}

\begin{proof}
	Calculating directly, we have that
	\begin{align}
		F_{1}^{3}&=\left( \frac{F+\bar{F}}{2}\right) ^3=\frac{1}{4}\left(\mathrm{Re}\left( F^3\right) +3\mathrm{Re}\left( F^2\bar{F}\right) \right), \\
		F_{2}^{3}&=\left( \frac{F-\bar{F}}{2\mathrm{i}}\right) ^3=\frac{1}{4}\left(-\mathrm{Im}\left(F^3\right) +3\mathrm{Im}\left( F^2\bar{F}\right)  \right),\\
		F_{1}^2F_{2}&=\left( \frac{F+\bar{F}}{2}\right) ^2\left( \frac{F-\bar{F}}{2\mathrm{i}}\right)=\frac{1}{4}\left(\mathrm{Im}\left( F^3\right) +\mathrm{Im}\left( F^2\bar{F}\right)  \right),\\
		F_{1}F_{2}^2&=\left( \frac{F+\bar{F}}{2}\right) \left( \frac{F-\bar{F}}{2\mathrm{i}}\right)^2=\frac{1}{4}\left(-\mathrm{Re}\left( F^3\right) +\mathrm{Re}\left( F^2\bar{F}\right)  \right),
	\end{align}
	where $\mathrm{Re}(z)$ and $\mathrm{Im}(z)$ denote the real and imaginary parts of a complex number $z$, respectively. Thus,
	\begin{align}
		\left| \mathrm{E}\left[ F_{1}^{3}\right] \right| & =\frac{1}{4}\left| \mathrm{Re}\left( \mathrm{E}\left[ F^3\right] \right) +3\mathrm{Re}\left( \mathrm{E}\left[F^2\bar{F}\right]\right)   \right| ,\\
		\left| \mathrm{E}\left[ F_{2}^{3}\right] \right| & =\frac{1}{4}\left| \mathrm{Im}\left( \mathrm{E}\left[ F^3\right] \right) -3\mathrm{Im}\left( \mathrm{E}\left[F^2\bar{F}\right]\right)   \right|, \\
		\left| \mathrm{E}\left[F_{1}^{2}F_{2}\right] \right| & =\frac{1}{4}\left| \mathrm{Im}\left( \mathrm{E}\left[ F^3\right] \right) +\mathrm{Im}\left( \mathrm{E}\left[F^2\bar{F}\right] \right)  \right|, \\
		\left| \mathrm{E}\left[ F_{1}F_{2}^{2}\right] \right| & =\frac{1}{4}\left| \mathrm{Re}\left( \mathrm{E}\left[ F^3\right] \right) -\mathrm{Re}\left( \mathrm{E}\left[F^2\bar{F}\right]\right)   \right| .
	\end{align}
	Then
	\begin{align}
		&\left| \mathrm{E}\left[F_{1}^3 \right] \right|+	\left| \mathrm{E}\left[F_{2}^3 \right] \right| +\left| \mathrm{E}\left[F_{1}^2F_{2} \right] \right|+ \left| \mathrm{E}\left[F_{1}F_{2}^2 \right] \right|\\
		\leq &\, \frac{1}{2}\left( \left| \mathrm{Re}\left( \mathrm{E}\left[ F^3\right] \right)\right|+\left| \mathrm{Im}\left( \mathrm{E}\left[ F^3\right] \right)\right| \right)+\left( \left| \mathrm{Re}\left( \mathrm{E}\left[F^2\bar{F}\right]\right)   \right|+\left| \mathrm{Im}\left( \mathrm{E}\left[F^2\bar{F}\right]\right)   \right|\right)  \\
		\leq &\, \sqrt{2}\left(\left|\mathrm{E}\left[F^3\right] \right|+ \left|\mathrm{E}\left[F^2\bar{F}\right] \right| \right) ,
	\end{align}
	where the first inequality is from  triangle inequality $|x\pm y|\leq |x|+|y|$ for $x,y \in \mathbb{R}$ and the second inequality is by the fact that $|x|+|y|\leq \sqrt{2}\sqrt{x^2+y^2}=\sqrt{2}|z|$ for a complex number $z=x+\mathrm{i}y$.
	
	Next, by triangle inequality $\big||x|-|y|\big|\leq |x\pm y|$ for $x,y \in \mathbb{R}$, we get that 
	\begin{align}
		&\left| \mathrm{E}\left[F_{1}^3 \right] \right|+	\left| \mathrm{E}\left[F_{2}^3 \right] \right| +\left| \mathrm{E}\left[F_{1}^2F_{2} \right] \right|+ \left| \mathrm{E}\left[F_{1}F_{2}^2 \right] \right|\\
		\geq &\, \frac{1}{4}\left(
		\Big| \left|\mathrm{Re}\left(  \mathrm{E}\left[ F^3\right] \right)\right|  -3\left| \mathrm{Re}\left( \mathrm{E}\left[F^2\bar{F}\right]\right)   \right|\Big| 
		+\Big| \left| \mathrm{Im}\left( \mathrm{E}\left[ F^3\right] \right)\right|  -3\left| \mathrm{Im}\left( \mathrm{E}\left[F^2\bar{F}\right]\right)   \right|\Big|\right. \\&\left.  
		\,+\Big| \left| \mathrm{Im}\left( \mathrm{E}\left[ F^3\right] \right) \right| -\left| \mathrm{Im}\left( \mathrm{E}\left[F^2\bar{F}\right] \right)  \right|\Big| 
		+\Big| \left| \mathrm{Re}\left( \mathrm{E}\left[ F^3\right] \right)\right|  -\left| \mathrm{Re}\left( \mathrm{E}\left[F^2\bar{F}\right]\right)   \right|\Big| 	\right)
	\end{align}
	Now we consider all cases:
	\begin{align}
		\textbf{Case1: } &\left|\mathrm{Re}\left(  \mathrm{E}\left[ F^3\right] \right)\right| \geq 3\left| \mathrm{Re}\left( \mathrm{E}\left[F^2\bar{F}\right]\right)   \right|, \left| \mathrm{Im}\left( \mathrm{E}\left[ F^3\right] \right)\right| \geq 3\left| \mathrm{Im}\left( \mathrm{E}\left[F^2\bar{F}\right]\right)   \right|. \\
		\textbf{Case2: } &\left|\mathrm{Re}\left(  \mathrm{E}\left[ F^3\right] \right)\right| \geq 3\left| \mathrm{Re}\left( \mathrm{E}\left[F^2\bar{F}\right]\right)   \right|, \\&\left| \mathrm{Im}\left( \mathrm{E}\left[F^2\bar{F}\right]\right)   \right|\leq \left| \mathrm{Im}\left( \mathrm{E}\left[ F^3\right] \right)\right| < 3\left| \mathrm{Im}\left( \mathrm{E}\left[F^2\bar{F}\right]\right)   \right|.\\
		\textbf{Case3: } &\left|\mathrm{Re}\left(  \mathrm{E}\left[ F^3\right] \right)\right| \geq 3\left| \mathrm{Re}\left( \mathrm{E}\left[F^2\bar{F}\right]\right)   \right|, \left| \mathrm{Im}\left( \mathrm{E}\left[ F^3\right] \right)\right|  < \left| \mathrm{Im}\left( \mathrm{E}\left[F^2\bar{F}\right]\right)   \right|.\\
		\textbf{Case4: } &\left| \mathrm{Re}\left( \mathrm{E}\left[F^2\bar{F}\right]\right)   \right|\leq \left|\mathrm{Re}\left(  \mathrm{E}\left[ F^3\right] \right)\right| < 3\left| \mathrm{Re}\left( \mathrm{E}\left[F^2\bar{F}\right]\right)   \right|,\\ &\left| \mathrm{Im}\left( \mathrm{E}\left[ F^3\right] \right)\right| \geq 3\left| \mathrm{Im}\left( \mathrm{E}\left[F^2\bar{F}\right]\right)   \right|.\\
		\textbf{Case5: } &\left| \mathrm{Re}\left( \mathrm{E}\left[F^2\bar{F}\right]\right)   \right|\leq \left|\mathrm{Re}\left(  \mathrm{E}\left[ F^3\right] \right)\right| < 3\left| \mathrm{Re}\left( \mathrm{E}\left[F^2\bar{F}\right]\right)   \right|, \\&\left| \mathrm{Im}\left( \mathrm{E}\left[F^2\bar{F}\right]\right)   \right|\leq \left| \mathrm{Im}\left( \mathrm{E}\left[ F^3\right] \right)\right| < 3\left| \mathrm{Im}\left( \mathrm{E}\left[F^2\bar{F}\right]\right)   \right|.\\
		\textbf{Case6: } &\left| \mathrm{Re}\left( \mathrm{E}\left[F^2\bar{F}\right]\right)   \right|\leq \left|\mathrm{Re}\left(  \mathrm{E}\left[ F^3\right] \right)\right| < 3\left| \mathrm{Re}\left( \mathrm{E}\left[F^2\bar{F}\right]\right)   \right|, \\&\left| \mathrm{Im}\left( \mathrm{E}\left[ F^3\right] \right)\right|  < \left| \mathrm{Im}\left( \mathrm{E}\left[F^2\bar{F}\right]\right)   \right|.\\
		\textbf{Case7: }& \left|\mathrm{Re}\left(  \mathrm{E}\left[ F^3\right] \right)\right|  < \left| \mathrm{Re}\left( \mathrm{E}\left[F^2\bar{F}\right]\right)   \right|, \left| \mathrm{Im}\left( \mathrm{E}\left[ F^3\right] \right)\right| \geq 3\left| \mathrm{Im}\left( \mathrm{E}\left[F^2\bar{F}\right]\right)   \right|.\\
		\textbf{Case8: }& \left|\mathrm{Re}\left(  \mathrm{E}\left[ F^3\right] \right)\right|  < \left| \mathrm{Re}\left( \mathrm{E}\left[F^2\bar{F}\right]\right)   \right|,\\ &\left| \mathrm{Im}\left( \mathrm{E}\left[F^2\bar{F}\right]\right)   \right|\leq \left| \mathrm{Im}\left( \mathrm{E}\left[ F^3\right] \right)\right| < 3\left| \mathrm{Im}\left( \mathrm{E}\left[F^2\bar{F}\right]\right)   \right|.\\
		\textbf{Case9: } &\left|\mathrm{Re}\left(  \mathrm{E}\left[ F^3\right] \right)\right|  < \left| \mathrm{Re}\left( \mathrm{E}\left[F^2\bar{F}\right]\right)   \right|, \left| \mathrm{Im}\left( \mathrm{E}\left[ F^3\right] \right)\right|  < \left| \mathrm{Im}\left( \mathrm{E}\left[F^2\bar{F}\right]\right)   \right|.
	\end{align}
	In \textbf{Case1},
	\begin{align}
		&\Big| \left|\mathrm{Re}\left(  \mathrm{E}\left[ F^3\right] \right)\right|  -3\left| \mathrm{Re}\left( \mathrm{E}\left[F^2\bar{F}\right]\right)   \right|\Big| 
		+\Big| \left| \mathrm{Im}\left( \mathrm{E}\left[ F^3\right] \right)\right|  -3\left| \mathrm{Im}\left( \mathrm{E}\left[F^2\bar{F}\right]\right)   \right|\Big| \\&\,+\Big| \left| \mathrm{Im}\left( \mathrm{E}\left[ F^3\right] \right) \right| -\left| \mathrm{Im}\left( \mathrm{E}\left[F^2\bar{F}\right] \right)  \right|\Big| 
		+\Big| \left| \mathrm{Re}\left( \mathrm{E}\left[ F^3\right] \right)\right|  -\left| \mathrm{Re}\left( \mathrm{E}\left[F^2\bar{F}\right]\right)   \right|\Big| \\
		=&\, \left|\mathrm{Re}\left(  \mathrm{E}\left[ F^3\right] \right)\right|  -3\left| \mathrm{Re}\left( \mathrm{E}\left[F^2\bar{F}\right]\right)   \right|+ \left| \mathrm{Im}\left( \mathrm{E}\left[ F^3\right] \right)\right|  -3\left| \mathrm{Im}\left( \mathrm{E}\left[F^2\bar{F}\right]\right)   \right|\\&\,+\left| \mathrm{Im}\left( \mathrm{E}\left[ F^3\right] \right) \right| -\left| \mathrm{Im}\left( \mathrm{E}\left[F^2\bar{F}\right] \right)  \right|+\left| \mathrm{Re}\left( \mathrm{E}\left[ F^3\right] \right)\right|  -\left| \mathrm{Re}\left( \mathrm{E}\left[F^2\bar{F}\right]\right)   \right|
		\\=&\, 2\left(\left| \mathrm{Re}\left( \mathrm{E}\left[ F^3\right] \right)\right|+\left| \mathrm{Im}\left( \mathrm{E}\left[ F^3\right] \right)\right| \right) -4\left(\left| \mathrm{Re}\left( \mathrm{E}\left[F^2\bar{F}\right]\right)   \right|+\left| \mathrm{Im}\left( \mathrm{E}\left[F^2\bar{F}\right]\right)   \right| \right) \\
		\geq&\, \frac{1}{2}\left(\left| \mathrm{Re}\left( \mathrm{E}\left[ F^3\right] \right)\right|+\left| \mathrm{Im}\left( \mathrm{E}\left[ F^3\right] \right)\right| \right) +\frac{1}{2}\left(\left| \mathrm{Re}\left( \mathrm{E}\left[F^2\bar{F}\right]\right)   \right|+\left| \mathrm{Im}\left( \mathrm{E}\left[F^2\bar{F}\right]\right)   \right| \right)\\
		\geq&\,\frac{1}{2}\left(\left|\mathrm{E}\left[F^3\right] \right|+ \left|\mathrm{E}\left[F^2\bar{F}\right] \right| \right),
	\end{align}
	where the penultimate inequality is from the conditions $\left|\mathrm{Re}\left(  \mathrm{E}\left[ F^3\right] \right)\right| \geq 3\left| \mathrm{Re}\left( \mathrm{E}\left[F^2\bar{F}\right]\right)   \right|$ and $\left| \mathrm{Im}\left( \mathrm{E}\left[ F^3\right] \right)\right| \geq 3\left| \mathrm{Im}\left( \mathrm{E}\left[F^2\bar{F}\right]\right)   \right|$, and the last inequality is by the fact that $|x|+|y|\geq \sqrt{x^2+y^2}=|z|$ for a complex number $z=x+\mathrm{i}y$ with $x,y\in \mathbb{R}$.
	
	By a similar argument, we can also obtain that
	\begin{align}
		&\Big| \left|\mathrm{Re}\left(  \mathrm{E}\left[ F^3\right] \right)\right|  -3\left| \mathrm{Re}\left( \mathrm{E}\left[F^2\bar{F}\right]\right)   \right|\Big| 
		+\Big| \left| \mathrm{Im}\left( \mathrm{E}\left[ F^3\right] \right)\right|  -3\left| \mathrm{Im}\left( \mathrm{E}\left[F^2\bar{F}\right]\right)   \right|\Big| \\&  
		+\Big| \left| \mathrm{Im}\left( \mathrm{E}\left[ F^3\right] \right) \right| -\left| \mathrm{Im}\left( \mathrm{E}\left[F^2\bar{F}\right] \right)  \right|\Big| 
		+\Big| \left| \mathrm{Re}\left( \mathrm{E}\left[ F^3\right] \right)\right|  -\left| \mathrm{Re}\left( \mathrm{E}\left[F^2\bar{F}\right]\right)   \right|\Big| \\
		\geq&\,\frac{1}{2}\left(\left|\mathrm{E}\left[F^3\right] \right|+ \left|\mathrm{E}\left[F^2\bar{F}\right] \right| \right),
	\end{align}
	is valid in all other eight cases. Then the proof is completed.
\end{proof}
	
		As an example, we consider a complex Ornstein-Uhlenbeck process defined by the stochastic differential equation 
	\begin{equation}\label{SDE}
		\mathrm{d}Z_t=-\gamma Z_t\mathrm{d}t+ \mathrm{d}\zeta_t, \quad t\geq0,
	\end{equation}
	where $Z_t$ is a complex-valued process, $Z_0=0$, $\gamma \in\mathbb{C}$ is unknown, $\lambda:=\mathrm{Re}\;\gamma > 0$, and $\zeta_t$ is a complex Brownian motion. That is $\zeta_t =\frac{ B_t^1+\i B_t^2}{\sqrt2}$, where
	$(B_t^1, B_t^2)$ is a two-dimensional standard Brownian motion. Suppose that only one trajectory $\left( Z_t, 0\leq t\leq T\right) $ can be observed. By minimizing $\int_{0}^{T}\left|\dot{Z}_{t}+\gamma Z_{t}\right|^{2} \mathrm{d} t$, one can obtain a least squares estimator of $\gamma$ as follows,
	\begin{equation}
		\hat{\gamma}_{T}=-\frac{\int_{0}^{T} \bar{Z}_{t} \mathrm{d} Z_{t}}{\int_{0}^{T}\left|Z_{t}\right|^{2} \mathrm{d} t}=\gamma- \frac{\int_{0}^{T} \bar{Z}_{t} \mathrm{d} \zeta_{t}}{\int_{0}^{T}\left|Z_{t}\right|^{2} \mathrm{d} t}.
	\end{equation}
	In \cite{ChenHuWang2017}, Chen, Hu and Wang proved that $\sqrt{T}\left(\hat{\gamma}_{T}-\gamma\right)$ is asymptotically normal. Namely, as $ T \rightarrow \infty$,
	\begin{equation}
		\sqrt{T}\left[\hat{\gamma}_{T}-\gamma\right]=-\frac{\frac{1}{\sqrt{T}}\int_{0}^{T} \bar{Z}_{t} \mathrm{d} \zeta_{t}}{\frac{1}{T}\int_{0}^{T}\left|Z_{t}\right|^{2} \mathrm{d} t} \overset{d}{\rightarrow} \mathcal{N}_2\left(0, \lambda\mathrm{Id}_2\right),
	\end{equation}
	where $\mathrm{Id}_2$ denotes $2\times2$ identity matrix.
	They showed that denominator satisfies
	\begin{equation}
		\frac{1}{T}\int_{0}^{T}\left|Z_{t}\right|^{2} \mathrm{d} t\overset{a.s.}{\rightarrow}\frac{1}{2\lambda},
	\end{equation} 
	and for numerator $F_T:=\frac{1}{\sqrt{T}}\int_{0}^{T} \bar{Z}_{t} \mathrm{d} \zeta_{t}$,
	\begin{equation}\label{example}
		\left( F_{1,T},F_{2,T}\right)=\left( \mathrm{Re}\left( F_T\right) ,\mathrm{Im}\left( F_T\right)\right)  \overset{d}{\rightarrow}\mathcal{N}_2\left(0, \frac{1}{4\lambda}\mathrm{Id}_2\right).
	\end{equation} 
	Then the asymptotic normality of the estimator $\hat{\gamma}_{T}$ is obtained. One should note that, in \cite{ChenHuWang2017}, the complex Ornstein-Uhlenbeck process considered by Chen, Hu and Wang is driven by a complex fractional Brownian motion with Hurst parameter belonging to $\left[ \frac{1}{2},\frac{3}{4}\right) $. The case that noise is a complex fractional Brownian motion involves more complicated calculations and more precise estimations. Here, to demonstrate the availability of our techniques, we focus on the case in which the noise is a complex standard Brownian motion. Next we will derive that $\frac{1}{\sqrt{T}}$ is the optimal rate of convergence for the numerator $F_T$. We have no idea to handle the optimal rate of convergence for the statistic $\sqrt{T}\left[\hat{\gamma}_{T}-\gamma\right]$, although we conjecture that it is still $\frac{1}{\sqrt{T}}$. Note that Kim and Park in \cite{KIM2017413,KIM2017284} obtained that $\frac{1}{\sqrt{T}}$ is the optimal Berry-Esseen bound for normal approximation of the least squares estimator of the drift coefficient of the real-valued one-dimensional Ornstein-Uhlenbeck process driven by a standard Brownian motion. As they stated in \cite{KIM2017284}, in many situations encountered in statistics, one need to consider the rate of convergence for the sequences $F_n/G_n$ with $G_n > 0$ almost surely (such as $\sqrt{T}\left[\hat{\gamma}_{T}-\gamma\right]$). Therefore, we shall deal with the optimal rate of convergence for the statistic $\sqrt{T}\left[\hat{\gamma}_{T}-\gamma\right]$ in separate project.

	Define the Hilbert space $\mathfrak{H}=L^2\left( \left[ 0,+\infty\right) \right) $ with inner product $\left\langle  f,g\right\rangle_{\mathfrak{H}}= \int_{0}^{\infty}  f(t) g(t)  \mathrm{d} t$. We complexify $\mathfrak{H}$ in the usual way and denote by $\mathfrak{H}_{\mathbb{C}}$. For any $f,g\in\mathfrak{H}_{\mathbb{C}}$, $\left\langle  f,g\right\rangle_{\mathfrak{H}_{\mathbb{C}}}= \int_{0}^{\infty}  f(t) \overline{g(t)} \mathrm{d} t$. Given $f\in\mathfrak{H}_{\mathbb{C}}^{\odot a}\otimes\mathfrak{H}_{\mathbb{C}}^{\odot b}$, $g\in\mathfrak{H}_{\mathbb{C}}^{\odot c}\otimes\mathfrak{H}_{\mathbb{C}}^{\odot d}$, for $i=0,\dots,a\land d$, $j=0,\dots,b\land c$, the $(i,j)$-th contraction of $f$ and $g$ is the element of $\mathfrak{H}_{\mathbb{C}}^{\odot (a+c-i-j)}\otimes\mathfrak{H}_{\mathbb{C}}^{\odot (b+d-i-j)}$ defined by
	\begin{align}
		&f \otimes_{i, j} g\left(t_{1}, \ldots, t_{a+c-i-j} ; s_{1}, \ldots, s_{b+d-i-j}\right) \\
		=&\,\int_{\mathbb{R}_{+}^{2 l}} f\left(t_{1}, \ldots, t_{a-i}, u_{1}, \ldots, u_{i} ; s_{1} \ldots, s_{b-j}, v_{1} ,\ldots, v_{j}\right) \\
		&\qquad \quad g\left(t_{a-i+1}, \ldots, t_{p-l}, v_{1}, \ldots, v_{j}; s_{b-j+1}, \ldots, s_{q-l}, u_{1}, \ldots, u_{i}\right) \mathrm{d} \vec{u} \mathrm{d} \vec{v} ,
	\end{align}
	where $l=i+j, p=a+c, q=b+d, \vec{u}=\left(u_{1}, \ldots, u_{i}\right)$ and $\vec{v}=\left(v_{1}, \ldots, v_{j}\right)$.
	
	According to \eqref{SDE}, we know that
	\begin{equation}\label{F_T}
		F_T=\frac{1}{\sqrt{T}}\int_{0}^{T} \bar{Z}_{t} \mathrm{d} \zeta_{t}=\frac{1}{\sqrt{T}}\int_{0}^{T} \int_{0}^{T}e^{-\bar{\gamma}(t-s)}\mathbf{1}_{\left\lbrace 0\leq s\leq t\leq T\right\rbrace} \mathrm{d} \zeta_{t}\mathrm{d}\bar{\zeta}_s=I_{1,1}(\frac{1}{\sqrt{T}}\psi_{T}(t,s)),
	\end{equation}
	where
	\begin{equation}
		\psi_{T}(t,s)= e^{-\bar{\gamma}(t-s)}\mathbf{1}_{\left\lbrace 0\leq s\leq t\leq T\right\rbrace},
	\end{equation}
and $\mathbf{1}_{E}$ is the indicator function of a set $E$.
	Let
	\begin{equation}
		h_{T}(t,s)=\overline{\psi_{T}(s,t)}= e^{-\gamma(s-t)}\mathbf{1}_{\left\lbrace 0\leq t\leq s\leq T\right\rbrace},
	\end{equation}
	then
	\begin{equation}
		\bar{F_T}=I_{1,1}(\frac{1}{\sqrt{T}}h_{T}(t,s)).
	\end{equation}
	By isometry property of complex Wiener It\^o integral, we obtain that
	\begin{equation}
		\begin{aligned}
			\mathrm{E}\left[ F_T^2\right] &=\mathrm{E}\left[ F_T\bar{\bar{F_T}}\right] =\mathrm{E}(I_{1,1}(\frac{1}{\sqrt{T}}\psi_{T})\overline{I_{1,1}(\frac{1}{\sqrt{T}}h_T)})=\frac{1}{T}\left\langle \psi_{T},h_T \right\rangle _{\mathfrak{H}^{\otimes 2}_{\mathbb{C}}}\\
			&=\frac{1}{T}\int_{0}^{\infty}\int_{0}^{\infty} \psi_{T}\left(t,s \right)  \overline{h_{T}\left( t,s\right) }\mathrm{d}t\mathrm{d}s\\
			&=\frac{1}{T}\int_{0}^{\infty}\int_{0}^{\infty} e^{-\bar{\gamma}(t-s)}\mathbf{1}_{\left\lbrace 0\leq s\leq t\leq T\right\rbrace} e^{-\bar{\gamma}(s-t)}\mathbf{1}_{\left\lbrace 0\leq t\leq s\leq T\right\rbrace}\mathrm{d}t\mathrm{d}s=0,
		\end{aligned}
	\end{equation}
	and
	\begin{equation}
		\begin{aligned}
			\mathrm{E}\left[ |F_T|^2\right] &=\mathrm{E}\left[ F_T\bar{F_T}\right] =\mathrm{E}(I_{1,1}(\frac{1}{\sqrt{T}}\psi_{T})\overline{I_{1,1}(\frac{1}{\sqrt{T}}\psi_{T})})=\frac{1}{T}\left\langle \psi_{T},\psi_{T} \right\rangle _{\mathfrak{H}^{\otimes 2}_{\mathbb{C}}}\\
			&=\frac{1}{T}\int_{0}^{\infty}\int_{0}^{\infty} \psi_{T}\left(t,s \right)  \overline{\psi_{T}\left( t,s\right) }\mathrm{d}t\mathrm{d}s
			\\&=\frac{1}{T}\int_{0}^{\infty}\int_{0}^{\infty} e^{-\bar{\gamma}(t-s)}\mathbf{1}_{\left\lbrace 0\leq s\leq t\leq T\right\rbrace} e^{-\gamma(t-s)}\mathbf{1}_{\left\lbrace 0\leq s\leq t\leq T\right\rbrace}\mathrm{d}t\mathrm{d}s
			\\&=\frac{1}{T}\int_{0}^{T}\int_{0}^{t}e^{-2\lambda(t-s)}\mathrm{d}s\mathrm{d}t=\frac{1}{2\lambda}+\frac{1}{4\lambda^2T}e^{-2\lambda T}-\frac{1}{4\lambda^2T}
			\\&\rightarrow\frac{1}{2\lambda}\text{ as }T\rightarrow\infty.  
		\end{aligned}
	\end{equation}
	Since $\lim\limits_{T\rightarrow \infty}\left( 1+\frac{1}{2\lambda T}e^{-2\lambda T}-\frac{1}{2\lambda T}\right)=1 $, for sufficient large $T$, $1+\frac{1}{2\lambda T}e^{-2\lambda T}-\frac{1}{2\lambda T}>0$. Consider 
	\begin{equation}
		F_T^{'}=\left(1+\frac{1}{2\lambda T}e^{-2\lambda T}-\frac{1}{2\lambda T} \right) ^{-\frac{1}{2}}F_T.
	\end{equation}
	Then the covariance matrix of $F_T^{'}$ is equal to $\frac{1}{4\lambda}\mathrm{Id}_2$. Now we consider the optimal rate of convergence of $F_T^{'}$ to a complex normal random variable $Z$ with the covariance matrix $\frac{1}{4\lambda}\mathrm{Id}_2$ under the distance $\rho(F_T^{'},Z)$ as $T\rightarrow\infty$.
	
	\begin{Thm}
		$F_T^{'}$ converges in distribution to a complex normal random variable $Z$ with the covariance matrix $\frac{1}{4\lambda}\mathrm{Id}_2$ and there exist two finite constants $0<c_1<c_{2}$ not depending on $T$ such that for $T$ large enough,
		\begin{equation}
			c_1 \frac{1}{\sqrt{T}} \leq \rho\left(F_T^{'}, Z\right) \leq c_{2} \frac{1}{\sqrt{T}}.
		\end{equation}
	\end{Thm}
	
	\begin{proof}
		By Theorem \ref{Main result2}, it suffices to show that 
		\begin{equation}
			M^{'}\left( \left(1+\frac{1}{2\lambda T}e^{-2\lambda T}-\frac{1}{2\lambda T} \right) ^{-\frac{1}{2}}F_T\right) \asymp \frac{1}{\sqrt{T}}.
		\end{equation}
		Equivalently, we need to prove that 
		\begin{align}\label{M(F_T)}
			&M^{'}(F_T)\\=&\,\max\left\lbrace \left| \mathrm{E}\left[F_T^3 \right] \right|, \left|\mathrm{E}\left[F_T^2\bar{F_T} \right] \right|  ,\mathrm{E}\left[\left| F_T \right| ^4 \right]-2\left(\mathrm{E}\left[ \left| F_T \right| ^2\right]  \right) ^2-\left| \mathrm{E}\left[ F_T^2\right] \right| ^2 \right\rbrace \\ \asymp& \,\frac{1}{\sqrt{T}}.
		\end{align}
		Combining the following Lemma \ref{3} and Lemma \ref{4}, we get \eqref{M(F_T)}. Then the proof is finished.
		
	\end{proof}

	\begin{Lemma}\label{3}
		$F_T$ is defined as \eqref{F_T}, then
		\begin{equation}
			\left| \mathrm{E}\left[F_T^3 \right] \right|=0, \quad \left|\mathrm{E}\left[F_T^2\bar{F_T} \right] \right|  \asymp \frac{1}{\sqrt{T}}.
		\end{equation}
	\end{Lemma}
	
	\begin{proof}
		According to the product formula of complex Wiener-It\^o integral \eqref{complex product}, we obtain that 
		\begin{align}
			F_T^2&=I_{1,1}\left( \frac{1}{\sqrt{T}}\psi_{T}\right) I_{1,1}\left( \frac{1}{\sqrt{T}}\psi_{T}\right) \\
			&=\sum_{i=0}^{1}\sum_{j=0}^{1}\binom{1}{i}^2\binom{1}{j}^2 i!j!I_{2-i-j,2-i-j}\left(\frac{1}{\sqrt{T}}\psi_{T}\otimes_{i, j}\frac{1}{\sqrt{T}}\psi_{T} \right) \\
			&=I_{2,2}\left(\frac{1}{\sqrt{T}}\psi_{T}\otimes\frac{1}{\sqrt{T}}\psi_{T} \right)+I_{1,1}\left(\frac{1}{\sqrt{T}}\psi_{T}\otimes_{1,0}\frac{1}{\sqrt{T}}\psi_{T} \right)\\&\quad+I_{1,1}\left(\frac{1}{\sqrt{T}}\psi_{T}\otimes_{0,1}\frac{1}{\sqrt{T}}\psi_{T} \right)+\frac{1}{\sqrt{T}}\psi_{T}\otimes_{1,1}\frac{1}{\sqrt{T}}\psi_{T},
		\end{align}
		\begin{align}
			F_T^3&=F_T^2F_T\\&=\sum_{i=0}^{1}\sum_{j=0}^{1}\binom{2}{i}\binom{1}{i}\binom{2}{j}\binom{1}{j} i!j!I_{3-i-j,3-i-j}\left(\left( \frac{1}{\sqrt{T}}\psi_{T}\otimes\frac{1}{\sqrt{T}}\psi_{T}\right) \otimes_{i, j}\frac{1}{\sqrt{T}}\psi_{T} \right) \\&\quad+ \sum_{i=0}^{1}\sum_{j=0}^{1}\binom{1}{i}^2\binom{1}{j}^2 i!j!I_{2-i-j,2-i-j}\left(\left( \frac{1}{\sqrt{T}}\psi_{T}\otimes_{1,0}\frac{1}{\sqrt{T}}\psi_{T}\right) \otimes_{i, j}\frac{1}{\sqrt{T}}\psi_{T} \right)\\&\quad+ \sum_{i=0}^{1}\sum_{j=0}^{1}\binom{1}{i}^2\binom{1}{j}^2 i!j!I_{2-i-j,2-i-j}\left(\left( \frac{1}{\sqrt{T}}\psi_{T}\otimes_{0,1}\frac{1}{\sqrt{T}}\psi_{T}\right) \otimes_{i, j}\frac{1}{\sqrt{T}}\psi_{T} \right)\\&\quad+\left( \frac{1}{\sqrt{T}}\psi_{T}\otimes_{1,1}\frac{1}{\sqrt{T}}\psi_{T}\right) I_{1,1}\left( \frac{1}{\sqrt{T}}\psi_{T}\right),
		\end{align}
		and 
		\begin{align}
			F_T^2\bar{F_T}&=\sum_{i=0}^{1}\sum_{j=0}^{1}\binom{2}{i}\binom{1}{i}\binom{2}{j}\binom{1}{j} i!j!I_{3-i-j,3-i-j}\left(\left( \frac{1}{\sqrt{T}}\psi_{T}\otimes\frac{1}{\sqrt{T}}\psi_{T}\right) \otimes_{i, j}\frac{1}{\sqrt{T}}h_{T} \right) \\&\quad+ \sum_{i=0}^{1}\sum_{j=0}^{1}\binom{1}{i}^2\binom{1}{j}^2 i!j!I_{2-i-j,2-i-j}\left(\left( \frac{1}{\sqrt{T}}\psi_{T}\otimes_{1,0}\frac{1}{\sqrt{T}}\psi_{T}\right) \otimes_{i, j}\frac{1}{\sqrt{T}}h_{T} \right)\\&\quad+ \sum_{i=0}^{1}\sum_{j=0}^{1}\binom{1}{i}^2\binom{1}{j}^2 i!j!I_{2-i-j,2-i-j}\left(\left( \frac{1}{\sqrt{T}}\psi_{T}\otimes_{0,1}\frac{1}{\sqrt{T}}\psi_{T}\right) \otimes_{i, j}\frac{1}{\sqrt{T}}h_{T} \right)\\&\quad+\left( \frac{1}{\sqrt{T}}\psi_{T}\otimes_{1,1}\frac{1}{\sqrt{T}}\psi_{T}\right) I_{1,1}\left( \frac{1}{\sqrt{T}}h_{T}\right) .
		\end{align}
		Taking expectation, we have that
		\begin{align}
			\mathrm{E}\left[ F_T^3\right] &=\left( \frac{1}{\sqrt{T}}\psi_{T}\otimes_{1,0}\frac{1}{\sqrt{T}}\psi_{T}\right) \otimes_{1, 1}\frac{1}{\sqrt{T}}\psi_{T}+\left( \frac{1}{\sqrt{T}}\psi_{T}\otimes_{0,1}\frac{1}{\sqrt{T}}\psi_{T}\right) \otimes_{1, 1}\frac{1}{\sqrt{T}}\psi_{T}\\
			&=2\left( \frac{1}{\sqrt{T}}\psi_{T}\otimes_{1,0}\frac{1}{\sqrt{T}}\psi_{T}\right) \otimes_{1, 1}\frac{1}{\sqrt{T}}\psi_{T}\\
			&=2\int_{0}^{\infty}\int_{0}^{\infty} \frac{1}{T^{3/2}}\mathbf{1}_{\left\lbrace 0\leq s\leq t\leq T\right\rbrace}\left(t-s \right) e^{-\bar{\gamma}(t-s)} e^{-\bar{\gamma}(s-t)}\mathbf{1}_{\left\lbrace 0\leq t\leq s\leq T\right\rbrace}\mathrm{d}s\mathrm{d}t\\
			&=0,
		\end{align}
		and 
		\begin{align}
			\mathrm{E}\left[ 	F_T^2\bar{F_T}\right] &=\left( \frac{1}{\sqrt{T}}\psi_{T}\otimes_{1,0}\frac{1}{\sqrt{T}}\psi_{T}\right) \otimes_{1, 1}\frac{1}{\sqrt{T}}h_{T}+\left(\frac{1}{\sqrt{T}}\psi_{T}\otimes_{0,1}\frac{1}{\sqrt{T}}\psi_{T}\right) \otimes_{1, 1}\frac{1}{\sqrt{T}}h_{T}\\
			&=2\left( \frac{1}{\sqrt{T}}\psi_{T}\otimes_{1,0}\frac{1}{\sqrt{T}}\psi_{T}\right) \otimes_{1, 1}\frac{1}{\sqrt{T}}h_{T}\\
			&=2\int_{0}^{\infty}\int_{0}^{\infty} \frac{1}{T^{3/2}}\mathbf{1}_{\left\lbrace 0\leq s\leq t\leq T\right\rbrace}\left(t-s \right) e^{-\bar{\gamma}(t-s)}e^{-\gamma(t-s)}\mathbf{1}_{\left\lbrace 0\leq s\leq t\leq T\right\rbrace} \mathrm{d}s\mathrm{d}t\\
			&=\frac{2}{T^{3/2}}\int_{0}^{T}\int_{0}^{t} \left(t-s \right) e^{-2\lambda(t-s)} \mathrm{d}s\mathrm{d}t\\
			&=\frac{2}{T^{3/2}}\int_{0}^{T}\int_{0}^{t} s e^{-2\lambda s} \mathrm{d}s\mathrm{d}t\\
			&=\frac{1}{2\lambda^2\sqrt{T}}e^{-2\lambda T}\left( 1+\frac{1}{\lambda T}\right)-\frac{1}{2\lambda^3 T^{3/2}}+\frac{1}{4\lambda^2\sqrt{ T}}\\
			&\asymp \frac{1}{\sqrt{T}}.
		\end{align}
		Then we get the conclusion.
	\end{proof}

	\begin{Lemma}\label{4}
		$F_T$ is defined as \eqref{F_T}, then
		\begin{equation}
			\mathrm{E}\left[\left| F_T \right| ^4 \right]-2\left(\mathrm{E}\left[ \left| F_T \right| ^2\right]  \right) ^2-\left| \mathrm{E}\left[ F_T^2\right] \right| ^2\asymp\frac{1}{T}.
		\end{equation}
	\end{Lemma}
	
	\begin{proof}
		\cite[Lemma 2.3]{ChenHuWang2017} shows that
		\begin{align}
			&\mathrm{E}\left[\left| F_T \right| ^4 \right]-2\left(\mathrm{E}\left[ \left| F_T \right| ^2\right]  \right) ^2-\left| \mathrm{E}\left[ F_T^2\right] \right| ^2\\
			=&\, \left\|\left(\frac{1}{\sqrt{T}} \psi_{T}\right)  \otimes_{0, 1} \left(\frac{1}{\sqrt{T}} \psi_{T}\right)\right\|_{\mathfrak{H}_{\mathbb{C}}^{\otimes2}}^{2} + \left\|\left(\frac{1}{\sqrt{T}} \psi_{T}\right)  \otimes_{1, 0} \left(\frac{1}{\sqrt{T}} \psi_{T}\right)\right\|_{\mathfrak{H}_{\mathbb{C}}^{\otimes2}}^{2} \\
			&\,+ \left\|\left(\frac{1}{\sqrt{T}} \psi_{T}\right)  \otimes_{0, 1} \left(\frac{1}{\sqrt{T}} h_{T}\right)+\left(\frac{1}{\sqrt{T}} \psi_{T}\right)  \otimes_{1, 0} \left(\frac{1}{\sqrt{T}} h_{T}\right)\right\|_{\mathfrak{H}_{\mathbb{C}}^{\otimes2}}^{2}.
		\end{align}
		Calculating directly, we get that 
		\begin{align}
			\left(\frac{1}{\sqrt{T}} \psi_{T}\right)  \otimes_{0, 1} \left(\frac{1}{\sqrt{T}} \psi_{T}\right)(t,s)&=\left(\frac{1}{\sqrt{T}} \psi_{T}\right)  \otimes_{1, 0} \left(\frac{1}{\sqrt{T}} \psi_{T}\right)(t,s)\\
			&=\frac{1}{T}\int_{0}^{\infty}\psi_{T}(t,u)\psi_{T}(u,s)\mathrm{d}u\\
			&=\frac{1}{T}\int_{0}^{\infty}e^{-\bar{\gamma}(t-u)}\mathbf{1}_{\left\lbrace 0\leq u\leq t\leq T\right\rbrace}e^{-\bar{\gamma}(u-s)}\mathbf{1}_{\left\lbrace 0\leq s\leq u\leq T\right\rbrace}\mathrm{d}u\\
			&=\frac{1}{T}\mathbf{1}_{\left\lbrace 0\leq s\leq t\leq T\right\rbrace}\int_{s}^{t}e^{-\bar{\gamma}(t-s)}\mathrm{d}u\\
			&=\frac{1}{T}\mathbf{1}_{\left\lbrace 0\leq s\leq t\leq T\right\rbrace}\left(t-s \right) e^{-\bar{\gamma}(t-s)},
		\end{align}
		\begin{align}
			\left(\frac{1}{\sqrt{T}} \psi_{T}\right)  \otimes_{0, 1} \left(\frac{1}{\sqrt{T}} h_{T}\right)\left(t,s \right)&=\frac{1}{T}\int_{0}^{\infty}\psi_{T}(t,u)h_{T}(u,s)\mathrm{d}u\\
			&=\frac{1}{T}\int_{0}^{\infty}e^{-\bar{\gamma}(t-u)}\mathbf{1}_{\left\lbrace 0\leq u\leq t\leq T\right\rbrace}e^{-\gamma(s-u)}\mathbf{1}_{\left\lbrace 0\leq u\leq s\leq T\right\rbrace}\mathrm{d}u\\
			&=\frac{1}{T}\mathbf{1}_{\left\lbrace 0\leq s, t\leq T\right\rbrace}e^{-\bar{\gamma}t-\gamma s}\int_{0}^{t\wedge s}e^{2\lambda u }\mathrm{d}u\\
			&=\frac{1}{2\lambda T}\mathbf{1}_{\left\lbrace 0\leq s, t\leq T\right\rbrace}e^{-\bar{\gamma}t-\gamma s}\left( e^{2\lambda\left(t \wedge s \right)  }-1\right),
		\end{align}
		and
		\begin{align}
			\left(\frac{1}{\sqrt{T}} \psi_{T}\right)  \otimes_{1, 0} \left(\frac{1}{\sqrt{T}} h_{T}\right)\left(t,s \right) &=\frac{1}{T}\int_{0}^{\infty}\psi_{T}(u,s)h_{T}(t,u)\mathrm{d}u\\
			&=\frac{1}{T}\int_{0}^{\infty}e^{-\bar{\gamma}(u-s)}\mathbf{1}_{\left\lbrace 0\leq s\leq u\leq T\right\rbrace}e^{-\gamma(u-t)}\mathbf{1}_{\left\lbrace 0\leq t\leq u\leq T\right\rbrace}\mathrm{d}u\\
			&=\frac{1}{T}\mathbf{1}_{\left\lbrace 0\leq s, t\leq T\right\rbrace}e^{\bar{\gamma}s+\gamma t}\int_{t\vee s}^{T}e^{-2\lambda u }\mathrm{d}u\\
			&=\frac{1}{2\lambda T}\mathbf{1}_{\left\lbrace 0\leq s, t\leq T\right\rbrace}e^{\bar{\gamma}s+\gamma t}\left( e^{-2\lambda\left(t \vee s \right)  }-e^{-2\lambda T}\right) .
		\end{align}
		Then 
		\begin{align}
			&\left\|\left(\frac{1}{\sqrt{T}} \psi_{T}\right)  \otimes_{0, 1} \left(\frac{1}{\sqrt{T}} \psi_{T}\right)\right\|_{\mathfrak{H}_{\mathbb{C}}^{\otimes2}}^{2} \\=&\,\frac{1}{T^2}\int_{0}^{\infty}\int_{0}^{\infty} \mathbf{1}_{\left\lbrace 0\leq s\leq t\leq T\right\rbrace}\left(t-s \right)^2 e^{-\gamma(t-s)}e^{-\bar{\gamma}(t-s)} \mathrm{d}t\mathrm{d}s\\
			=&\,\frac{1}{T^2}\int_{0}^{T}\int_{0}^{t} \left(t-s \right)^2 e^{-2\lambda(t-s)} \mathrm{d}s\mathrm{d}t=\frac{1}{T^2}\int_{0}^{T}\int_{0}^{t} s^2 e^{-2\lambda s} \mathrm{d}s\mathrm{d}t\\
			=&\,\frac{1}{2\lambda^2}e^{-2\lambda T}\left( \frac{1}{2}+\frac{1}{\lambda T}+\frac{3}{4\lambda^2 T^2}\right)-\frac{3}{8\lambda^4 T^2}+\frac{1}{4\lambda^3 T}\asymp \frac{1}{T},
			\end{align}
		\begin{equation}
			\left\|\left(\frac{1}{\sqrt{T}} \psi_{T}\right)  \otimes_{1, 0} \left(\frac{1}{\sqrt{T}} \psi_{T}\right)\right\|_{\mathfrak{H}_{\mathbb{C}}^{\otimes2}}^{2}=	\left\|\left(\frac{1}{\sqrt{T}} \psi_{T}\right)  \otimes_{0, 1} \left(\frac{1}{\sqrt{T}} \psi_{T}\right)\right\|_{\mathfrak{H}_{\mathbb{C}}^{\otimes2}}^{2}\asymp \frac{1}{T},
		\end{equation}
		and
		\begin{align}
			&\left\|\left(\frac{1}{\sqrt{T}} \psi_{T}\right)  \otimes_{0, 1} \left(\frac{1}{\sqrt{T}} h_{T}\right)+\left(\frac{1}{\sqrt{T}} \psi_{T}\right)  \otimes_{1, 0} \left(\frac{1}{\sqrt{T}} h_{T}\right)\right\|_{\mathfrak{H}_{\mathbb{C}}^{\otimes2}}^{2}\\
			=&\,\frac{1}{4\lambda^2 T^2}\int_{0}^{\infty}\int_{0}^{\infty}  \mathbf{1}_{\left\lbrace 0\leq s, t\leq T\right\rbrace}\left( e^{-\bar{\gamma}t-\gamma s}\left( e^{2\lambda\left(t \wedge s \right)  }-1\right)+e^{\bar{\gamma}s+\gamma t}\left( e^{-2\lambda\left(t \vee s \right)  }-e^{-2\lambda T}\right)  \right) \\&\qquad\qquad\qquad\qquad\qquad
			\left( e^{-\gamma t-\bar{\gamma} s}\left( e^{2\lambda\left(t \wedge s \right)  }-1\right)+e^{\gamma s+\bar{\gamma} t}\left( e^{-2\lambda\left(t \vee s \right)  }-e^{-2\lambda T}\right)  \right) 
			\mathrm{d}s\mathrm{d}t\\
			=&\,	\frac{1}{2\lambda^2 T^2}\int_{0}^{T}\int_{0}^{t}  \left( e^{-\bar{\gamma}t-\gamma s}\left( e^{\left( \gamma+\bar{\gamma}\right)  s }-1\right)+e^{\bar{\gamma}s+\gamma t}\left( e^{-\left( \gamma+\bar{\gamma}\right)  t }-e^{-2\lambda T}\right)  \right) \\&\qquad\qquad\qquad\qquad\qquad
			\left( e^{-\gamma t-\bar{\gamma} s}\left( e^{\left( \gamma+\bar{\gamma}\right)  s }-1\right)+e^{\gamma s+\bar{\gamma} t}\left( e^{-\left( \gamma+\bar{\gamma}\right)  t }-e^{-2\lambda T}\right)  \right) 
			\mathrm{d}s\mathrm{d}t\\
			=&\,\frac{1}{4\lambda^2}e^{-2\lambda T}\left( 2+\frac{8}{\lambda T}+\frac{5}{\lambda^2 T^2}+\frac{1}{2\lambda^2 T^2}e^{-2\lambda T}\right)-\frac{11}{8\lambda^4 T^2}+\frac{1}{\lambda^3 T}\\
			\asymp &\frac{1}{T}.
		\end{align}
		Then the proof is finished.
	\end{proof}

	\subsection{Application for Wiener-It\^o integrals with kernels of step functions}\label{Section 4.2}
	
	In \cite[Section 5.1]{Campese2013}, Campese proposed this counterexample to explain that his techniques sometimes fail to provide the optimal rate of convergence. In this section, for this example, we apply our conclusions to get the optimal rate of convergence with respect to the distance $\rho(\cdot,\cdot)$. Specifically, let $\mathfrak{H}=L^2([0,1), \mu)$, where $\mu$ is the Lebesgue measure on $[0,1)$, and partition $[0,1)$ into $N$ equidistant intervals $\alpha_1, \alpha_2, \ldots, \alpha_N$ where $\alpha_k=\left[\frac{k-1}{N}, \frac{k}{N}\right)$ for $k=1, \ldots, N$. Define $f \in \mathfrak{H}^{\odot 2}$ as
	\begin{equation}\label{step function}
		f(x, y)=\sum_{i, j=1}^N a_{i j} \mathbf{1}_{\alpha_i}(x) \mathbf{1}_{\alpha_j}(y),
	\end{equation}
	where $a_{i j}\in\mathbb{R}$, $a_{i j}=a_{j i}$ for $1\leq i,j\leq d$. It is obvious that $f$ is uniquely determined by the symmetric matrix $A=\left(a_{i j}\right)_{1 \leq i, j \leq N}$. If $g$ is another kernel of the type \eqref{step function}, given by a matrix $B=\left(b_{i j}\right)_{1 \leq i, j \leq N}$, we have
	$$
	\begin{aligned}
		\left(f \otimes_1 g\right)(x, y) & =\int_0^1\left(\sum_{i, j=1}^N a_{i j} \mathbf{1}_{\alpha_i}(x) \mathbf{1}_{\alpha_j}(t)\right)\left(\sum_{k, l=1}^N b_{k l} \mathbf{1}_{\alpha_k}(y) \mathbf{1}_{\alpha_l}(t)\right) \mathrm{d} \mu(t) \\
		& =\sum_{i, j, k=1}^N a_{i j} b_{k j} \mu\left(\alpha_j\right) \mathbf{1}_{\alpha_i}(x) \mathbf{1}_{\alpha_k}(y) \\
		& =\frac{1}{N} \sum_{i, j, k=1}^N a_{i j} b_{j k} \mathbf{1}_{\alpha_i}(x) \mathbf{1}_{\alpha_k}(y),
	\end{aligned}
	$$
	and
	$$
	\left(f \tilde{\otimes}_1 g\right)(x, y)=\frac{1}{2 N} \sum_{i, j, k=1}^N\left(a_{i j} b_{j k}+a_{ kj} b_{ji}\right) \mathbf{1}_{\alpha_i}(x) \mathbf{1}_{\alpha_k}(y) .
	$$
	Therefore, $f \otimes_1 g$ can be identified with the matrix $C=\frac{1}{N} A B$ and $f \tilde{\otimes}_1 g$ by $\frac{1}{2}\left(C+C^T\right)$. Similarly, one can show that
	\begin{equation}\label{contraction}
		\langle f, g\rangle_{\mathfrak{H}^{\otimes 2}}=\frac{1}{N^2}\langle A, B\rangle_{\mathrm{HS}}=\frac{\operatorname{tr}\left(A B^T\right)}{N^2}.
	\end{equation}
	For simplicity, we fix $d=2$. For $n \geq 1$, we define 2-dimensional random vectors $F_n=\left(I_2\left(f_{n,1}\right),  I_2\left(f_{n,2}\right)\right)$, where the kernels $f_{n,1}$ and $f_{ n,2}$ are given by $(3n)\times(3n)$ matrices $A_{n,1}=\sqrt{n}\begin{pmatrix}
		\mathbf{0}_n & \mathbf{0}_n & \mathds{1}_n\\
		\mathbf{0}_n & \mathbf{0}_n & \mathbf{0}_n\\
		\mathds{1}_n & \mathbf{0}_n & \mathbf{0}_n
		\end{pmatrix}$ and  $A_{n,2}=\sqrt{n}\begin{pmatrix}
		\mathbf{0}_n & \mathbf{0}_n & \mathbf{0}_n\\
		\mathbf{0}_n & \mathds{1}_n & \mathbf{0}_n\\
		\mathbf{0}_n & \mathbf{0}_n & \mathbf{0}_n
	\end{pmatrix}$, respectively. Here, we denote by $\mathbf{0}_n$ the $n \times n$ matrix with all entries equal to 0 and $\mathds{1}_n$ the $n \times n$ matrix with entries on anti-diagonal equal to 1 and other entries equal to 0.

	 According to \eqref{cumulant for second chaos} and \eqref{contraction}, for $1\leq i,j,k\leq 2$, we get that 
	\begin{align}
		\kappa_{e_i+e_j}(F_n)&=2\left\langle f_{n,i} , f_{n,j} \right\rangle _{\mathfrak{H}^{\otimes 2}}=\frac{2\mathrm{Tr}\left(A_{n,i}A_{n,j} \right) }{9n^2}=\begin{cases}
		\frac{4}{9}	,& i=j=1,\\
		\frac{2}{9}	,& i=j=2,\\
		0	,& i\neq j,
		\end{cases}\\
		\kappa_{e_i+e_j+e_k}(F_n)&=2^2\cdot2!\left\langle f_{n,i} \tilde{\otimes}_1 f_{n,j}, f_{n,k} \right\rangle _{\mathfrak{H}^{\otimes 2}}\\&=\frac{8\mathrm{Tr}\left(\frac{1}{6n}\left( A_{n,i}A_{n,j}+A_{n,j}A_{n,i}\right) A_{n,k} \right) }{9n^2}\\&=\frac{8\mathrm{Tr}\left(A_{n,i}A_{n,j}A_{n,k} \right) }{27n^3}=\begin{cases}
			\frac{8}{27n^{3/2}}	,& n \mbox{ is odd and } i=j=k=2,\\
			0	,& otherwise.
		\end{cases}.
	\end{align}
    By a similar argument, we know that
	\begin{equation}
		\kappa_{4e_i}(F_n)=2^3\cdot3!\frac{\mathrm{Tr}\left(A_{n,i}^4 \right) }{(3n)^4}=\begin{cases}
		\frac{32}{27n}	, &i=1,\\
		\frac{16}{27n}	, &i=2.
		\end{cases}
	\end{equation}
	Therefore, as $n\rightarrow\infty$, $F_n$ converges in distribution to a 2-dimensional normal random vector $Z\sim\mathcal{N}_2\left( 0,\frac{2}{9}\begin{pmatrix}
		2&0\\
		0&1
	\end{pmatrix}\right) $ by the multidimensional Fourth Moment Theorem (see \cite[Theorem 1]{Peccati2005}), and 
	\begin{equation}
		M(F_n)=\max\left\lbrace \sum_{|m|=3}\left|  \kappa_{m}(F_n)\right|, \sum_{i=1}^{d}\kappa_{4e_i}(F_n)\right\rbrace\asymp \frac{1}{n}.
	\end{equation}
Then we obtain the following theorem.
\begin{Thm}
	For $n\geq 1$, define $F_n$ as above. Then $F_n$ converges in distribution to $Z\sim\mathcal{N}_2\left( 0,\frac{2}{9}\begin{pmatrix}
		2&0\\
		0&1
	\end{pmatrix}\right) $ as $n\rightarrow\infty$, and there exist two finite constants $0<c_1<c_2$ not depending on $n$ such that for $n$ large enough,
\begin{equation}
	c_1 \frac{1}{n}\leq \rho(F_n,Z)< c_2\frac{1}{n}.
\end{equation}
\end{Thm}

	\subsection{Application for vector-valued Toeplitz quadratic functional }\label{Section 4.3}
	
	Let $X=\left(X_t\right)_{t \in \mathbb{R}}$ be a centered real-valued stationary Gaussian process with a covariance
	function $r(t): \mathbb{R} \rightarrow \mathbb{R}$ and a integrable and even spectral density $f(\lambda): \mathbb{R} \rightarrow \mathbb{R}$. This is, for every $u, t \in \mathbb{R}$, one has
	$$
	E\left(X_u X_{u+t}\right):=r(t)=\hat{f}(t):=\int_{-\infty}^{+\infty} e^{\mathrm{i} \lambda t} f(\lambda) \mathrm{d} \lambda,
	$$
	where $\hat{f}$ denotes its Fourier transform. We consider normalized random variable $$\tilde{Q}_{g,T}=\frac{Q_{g,T}-\mathrm{E}\left(Q_{g,T}\right)}{\sqrt{T}},$$
	where $Q_{g,T}$ is called Toeplitz quadratic functional of the process $X$ associated  with some integrable even function $g$ and $T>0$, defined as
	$$
	Q_{g,T}=\int_0^T\int_0^T \hat{g}(t-s) X(t) X(s) \mathrm{d} t \mathrm{d}s.
	$$
	Given $T>0$ and $\psi \in \mathrm{L}^1(\mathbb{R})$, we denote by $B_T(\psi)$ the truncated Toeplitz operator associated with $\psi$ and $T$, defined on $\mathrm{L}^2(\mathbb{R})$ as
	$$
	B_T(\psi)(u)(t)=\int_0^T u(x) \hat{\psi}(t-x) \mathrm{d} x, \quad t \in \mathbb{R}.
	$$
	Given $\psi, \gamma \in \mathrm{L}^1(\mathbb{R})$, let $B_T(\psi) B_T(\gamma)$ be the product of the two operators $B_T(\psi)$ and $B_T(\gamma)$. We denote by $\operatorname{Tr}(A)$ the trace of an operator $A$.
	
	We refer reader to \cite{Avram1988,Fox1987,Ginovian1994,Ginovyan2005,Ginovyan2007,Giraitis1990,Grenander1958} for the central limit theorems for Toeplitz quadratic functionals of discrete-time and continuous-time stationary Gaussian processes. Choosing even functions $g_1, \ldots, g_d \in L^1(\mathbb{R})$, we consider the random vector $G_T=\left(G_{1, T}, \ldots, G_{d, T}\right)$ defined by setting $G_{i, T}=\tilde{Q}_{g_i, T}$ for $1 \leq i \leq d$ and $T>0$.
	
	\begin{Thm}{\cite[Theorem 5.3]{Campese2013}}\label{Cumulant of Toeplitz quadratic functional}
		Let $m \in \mathbb{N}_0^d$ be a multi-index with $|m| \geq 2$ and elementary decomposition $\left\{l_1, \ldots, l_{|m|}\right\}$. For $1\leq i\leq |m|$, let $g_{l_i}=g_j$ if $l_i=e_j$, $1\leq j\leq d$. Then
		\begin{enumerate}[(i)]
			\item The cumulant $\kappa_m\left(G_T\right)$ is given by
			$$
			\kappa_m\left(G_T\right)=T^{-|m| / 2} 2^{|m|-1}(|m|-1) ! \operatorname{Tr}\left[B_T(f)^{|m|} \prod_{i=1}^{|m|} B_T\left(g_{l_i}\right)\right] .
			$$
			\item If $f \in L^1(\mathbb{R}) \cap L^{q_0}(\mathbb{R})$ and $g_i \in L^1(\mathbb{R}) \cap L^{q_i}(\mathbb{R})$ such that $1 / q_0+1 / q_i \leq 1 /|m|$ for $1 \leq i \leq d$, then
			$$
			\lim _{T \rightarrow \infty} T^{|m| / 2-1} \kappa_m\left(G_T\right)=2^{|m|-1}(|m|-1) !(2\pi)^{2|m|-1} \int_{-\infty}^{\infty} f^{|m|}(x) \prod_{i=1}^{|m|} g_{l_i}(x) \mathrm{d} x .
			$$
			\item If $f \in L^1(\mathbb{R}) \cap L^{q_0}(\mathbb{R})$ and $g_i \in L^1(\mathbb{R}) \cap L^{q_i}(\mathbb{R})$ such that $1 / q_0+1 / q_i \leq 1 / 2$ for $1 \leq i \leq d$, then
			$$
			G_T \overset{d}{\rightarrow}Z\sim \mathcal{N}_d(0, C), \quad T \rightarrow \infty,
			$$
			where the covariance matrix $C=\left(C_{i j}\right)_{1 \leq i, j \leq d}$ is given by
			$$
			C_{i j}=16\pi^3\int_0^{\infty} f^2(x) g_i(x) g_j(x) \mathrm{d} x .
			$$
		\end{enumerate}
	\end{Thm} 
	
	Suppose that $C$ is invertible. We denote by $C_T$ the covariance matrix of $G_T$. Then for $T$ large enough, $G_T$ is invertible. We now consider random vector $G_T^{'}=C^{\frac{1}{2}}C_T^{-\frac{1}{2}}G_T$. Note that each component $G_{i, T}=\tilde{Q}_{g_i, T}$ of $G_T$ can be represented as a double Wiener-It\^o integral with respect to $X$. Combining Theorem \ref{Main result1} and Theorem \ref{Cumulant of Toeplitz quadratic functional}, we obtain the optimal rate of convergence of $G_T^{'}$ to multivariate normal distribution $Z\sim\mathcal{N}_d(0, C)$ under the distance $\rho(G_T^{'}, Z)$ as $T\rightarrow\infty$. We point out that the optimal rate of convergence given in Theorem \ref{example 3} is more explicit compared to \cite[Proposition 5.3]{Campese2013}.
	
	\begin{Thm}\label{example 3}
		If $f \in L^1(\mathbb{R}) \cap L^{q_0}(\mathbb{R})$ and $g_i \in L^1(\mathbb{R}) \cap L^{q_i}(\mathbb{R})$ such that $1 / q_0+1 / q_i \leq 1 / 4$ for $1 \leq i \leq d$, then
		$G_T^{'} \overset{d}{\rightarrow} Z\sim\mathcal{N}_d(0, C)$ as $\rightarrow \infty$. Moreover,
		\begin{enumerate}[(i)]
			\item If $\int_{-\infty}^{\infty} f^{3}(x) \prod_{i=1}^{3} g_{l_i}(x) \mathrm{d} x\neq 0$ for some multi-index $m$ with $|m|=3$ and elementary decomposition $\left\{l_1, l_2, l_{3}\right\}$, then there exist two finite constants $0<c_1<c_2$ not depending on $T$ such that for $T$ large enough,
			\begin{equation}
				c_1\frac{1}{\sqrt{T}}\leq \rho(G_T^{'}, Z) \leq c_2\frac{1}{\sqrt{T}}.
			\end{equation}
			\item If $\int_{-\infty}^{\infty} f^{4}(x) g_{i}^4(x) \mathrm{d} x\neq 0$ for some $1\leq i\leq d$ and  $\int_{-\infty}^{\infty} f^{3}(x) \prod_{i=1}^{3} g_{l_i}(x) \mathrm{d} x= 0$ for any multi-index $m$ with $|m|=3$ and elementary decomposition $\left\{l_1, l_2, l_{3}\right\}$, then there exist two finite constants $0<c_1<c_2$ not depending on $T$ such that for $T$ large enough,
			\begin{equation}
				c_1\frac{1}{T}\leq \rho(G_T^{'}, Z) \leq c_2\frac{1}{T}.
			\end{equation}
		\end{enumerate}  
	\end{Thm}

\end{document}